\documentclass[12pt,reqno]{amsart}
\usepackage{amsmath,amssymb,extarrows}
\usepackage{url}
\usepackage{tikz,enumerate}
\usepackage{diagbox}
\usepackage{appendix}
\usepackage{epic}
\usepackage{comment}

\usepackage{float} 

\usepackage{cite}

\usepackage{hyperref}
\usepackage{array}

\usepackage{booktabs}

\allowdisplaybreaks[4]

\newtheorem{theorem}{Theorem}
\newtheorem{corollary}[theorem]{Corollary}
\newtheorem{conj}[theorem]{Conjecture}
\newtheorem{lemma}[theorem]{Lemma}

\theoremstyle{definition}

\theoremstyle{remark}
\newtheorem{rem}{Remark}
\numberwithin{equation}{section}
\numberwithin{theorem}{section}
\numberwithin{defn}{section}

\newcommand{\s}{\mathcal{S}}

\newcommand{\abs}[1]{\left\vert#1\right\vert}

\newcommand{\CT}{\mathrm{CT}}

\usepackage{graphicx} 

\begin{document}
\title[Modularity of tadpole Nahm sums]{Modularity of tadpole Nahm sums in ranks 4 and 5}
\author{Changsong Shi and Liuquan Wang}

\address[C.\ Shi]{School of Mathematics and Statistics, Wuhan University, Wuhan 430072, Hubei, People's Republic of China}
\email{changsong@whu.edu.cn}

\address[L.\ Wang]{School of Mathematics and Statistics, Wuhan University, Wuhan 430072, Hubei, People's Republic of China}
\email{wanglq@whu.edu.cn;mathlqwang@163.com}



\begin{abstract}
Around 2016, Calinescu, Milas and Penn conjectured that the rank $r$ Nahm sum associated with the $r\times r$ tadpole Cartan matrix is modular, and they provided a proof for $r=2$. The $r=3$ case was recently resolved by Milas and Wang. We prove this conjecture for the next cases $r=4,5$. We also prove the modularity of some companion Nahm sums by establishing the corresponding Rogers--Ramanujan type identities. A key new ingredient in our proofs is some rank reduction formulas which allow us to decompose higher rank tadpole Nahm sums to mixed products of some lower rank Nahm-type sums and theta functions.
\end{abstract}
\maketitle

\section{Introduction and main results}
Through this paper, we use the following standard $q$-series notation:
\begin{align*}
   & (a; q)_{n}:=\prod_{k=0}^{n-1}\left(1-a q^{k}\right), \quad(a ; q)_{\infty}:=\prod_{k=0}^{\infty}\left(1-a q^{k}\right), |q|<1, \\
& \left(a_{1}, \dots, a_{m} ; q\right)_{n}:=\left(a_{1} ; q\right)_{n} \cdots\left(a_{m} ; q\right)_{n}, \quad n \in \mathbb{N} \cup\{\infty\}.
\end{align*}

One of the central problems linking the theory of $q$-series and modular forms is to find all modular $q$-hypergeometric series. Along this direction, W.\ Nahm \cite{Nahm1994,Nahmconf,Nahm2007} considered a particular important class of series which were usually referred as Nahm sums. For a $r\times r$ rational positive definite matrix $A$, rational vector $B$ of length $r$ and rational scalar $C$, the rank $r$ Nahm sum corresponding to $(A,B,C)$ is defined as
\begin{align}
    f_{A,B,C}(q):=\sum_{n=(n_1,...,n_r)^\mathrm{T}\in(\mathbb{Z}_{\geq0})^r}\frac{q^{\frac{1}{2}n^\mathrm{T}An+n^\mathrm{T}B+C}}{(q;q)_{n_1}\cdots (q;q)_{n_r}}.
\end{align}
Nahm's problem is to find all modualr Nahm sums $f_{A,B,C}(q)$, and such $(A,B,C)$ is called as a \emph{modular triple}. The motivation of Nahm's study comes from physics as such modular Nahm sums are expected to be characters of some rational conformal field theories.

In the rank one case, Zagier \cite{Zagier} solved Nahm's problem by proving that there are exactly seven modular triples $(A,B,C)$ corresponding to
\begin{align}\label{eq-triple}
&(1/2,0,-1/40), ~~ (1/2,1/2,1/40), ~~(1,0,-1/48), ~~ (1,1/2,1/24), \nonumber \\
&(1,-1/2,1/24), ~~ (2,0,-1/60), ~~ (2,1,11/60).
\end{align}
The modularity of the triples with $A=1$ follows from one of Euler's $q$-exponential identities (see e.g.\ \cite[Corollary 2.2]{Andrews-book})
\begin{align}
    \sum_{n=0}^{\infty} \frac{q^{\frac{n^2-n}{2}} z^{n}}{(q ; q)_{n}}=(-z ; q)_{\infty}, \quad|z|<1. \label{euler-2}
\end{align}
The modularity of the other four triples were justified by Rogers' identities \cite[pp.\ 328,330,331]{Rogers1894}: for $a=0,1$,
\begin{align}
    \sum_{n=0}^\infty\frac{q^{n^2+an}}{(q;q)_n}
   & =
    \frac{1}{(q^{a+1},q^{4-a};q^5)_\infty}, \label{RR} \\
    \sum_{n=0}^\infty \frac{q^{n^2+an}}{(q^4;q^4)_n}&=\frac{1}{(-q^2;q^2)_\infty (q^{a+1},q^{4-a};q^5)_\infty}. \label{Rogers-id}
\end{align}
The identities in \eqref{RR} were later rediscovered by Ramanujan and hence are usually referred as Rogers--Ramanujan identities. They motivate poeple to search for similar sum-to-product $q$-hypergeometric series identities. It is worth mentioning that the Nahm sums in the left sides of \eqref{RR} correspond to two characters of the Lee--Yang model which agrees with Nahm's motivation.

When the rank $r\geq 2$, Nahm's problem becomes harder and is far from being solved. Through an extensive search, Zagier found 11 and 12 examples of possible modular triples in the cases of rank two and three, respectively. The modularity of these examples have all been proved by works of Zagier \cite{Zagier}, Vlasenko--Zwegers \cite{VZ}, Cherednik--Feigin \cite{Feigin}, Cao--Rosengren--Wang \cite{CRW} and Wang \cite{Wang-rank2,Wang-rank3}. Zagier also expected that there might exist some dual structure between modular triples. He \cite[p.\ 50, (f)]{Zagier} observed that if $(A,B,C)$ is a rank $r$ modular triple, then it is likely that the \emph{dual triple}
\begin{align}\label{eq-dual}
    (A^\star,B^\star,C^\star):=(A^{-1},A^{-1}B,\frac{1}{2}B^\mathrm{T}A^{-1}B-\frac{r}{24}-C)
\end{align}
is also a modular triple. Recently, Wang \cite{Wang-counter} provided some counterexamples to this observation.

It is now known that Zagier's list of rank two and three modular examples are incomplete. Vlasenko and Zwegers \cite{VZ} found some additional rank two modular triples. Recently, Cao and Wang \cite{Cao-Wang} discovered some new rank three modular triples through a lift-dual operation. It is difficult to give a complete classification of modular Nahm sums but finding families of modular Nahm sums with arbitrary rank is quite meaningful. For instance, as a generalization of the Rogers--Ramanujan identities, the Andrews--Gordon identities \cite{Andrews1974,Gordon1961} assert that: for integers $k,s$ such that $k\geq 2$ and $1\leq s \leq k$,
\begin{align}
\sum_{n_1,\dots,n_{k-1}\geq 0} \frac{q^{N_1^2+\cdots+N_{k-1}^2+N_s+\cdots +N_{k-1}}}{(q;q)_{n_1}(q;q)_{n_2}\cdots (q;q)_{n_{k-1}}}  =\frac{(q^s,q^{2k+1-s},q^{2k+1};q^{2k+1})_\infty}{(q;q)_\infty} \label{AG}
\end{align}
where if $j\leq k-1$, $N_j=n_j+\cdots+n_{k-1}$ and $N_k=0$. This generates a family of modular Nahm sums with arbitrary rank.

Around 2016, Calinescu, Milas and Penn \cite{CMP} considered another family of Nahm sums associated with the tadpole Cartan matrix $T_r=(a_{ij})_{r\times r}$ with entries defined as
\begin{align}
   & a_{rr}=1,\ a_{ii}=2 ~~ \text{for}~~ i<n, ~~ a_{ij}=-1 ~~\text{for} ~~ \abs{i-j}=1, \\
    &\text{and}~~  a_{ij}=0 ~~ \text{otherwise.}\nonumber
\end{align}
Here we agree that $T_1=(1)$. Let
\begin{align}
    \chi_r(x_1,\dots,x_r)=\chi_r(x_1,\dots,x_r;q):=\sum_{n=(n_1,\dots,n_r)^\mathrm{T}\in(\mathbb{Z}_{\geq0})^r}\frac{q^{\frac{1}{2}n^\mathrm{T}T_r n}x_1^{n_1}\cdots x_r^{n_r}}{(q;q)_{n_1}\cdots (q;q)_{n_r}}
\end{align}
be a generalized tadpole Nahm sum. Calinescu et al. \cite{CMP} proposed the following conjecture regarding its modularity.
\begin{conj}\label{conj-CMP}
(Cf.\ \cite[Conjecture 1]{CMP}.) For any $r\geq 2$ there exists a rational number $a$ such that $q^a\chi_r(1,1,\dots,1;q)$ is modular.
\end{conj}
The rank one case corresponds to the known modular triple $(1,0,-1/48)$. Calinescu et al. \cite{CMP} also proved the conjecture for $r=2$ which coincides with Zagier's third rank two example listed in \cite[Table 2]{Zagier}. Recently, Milas and Wang \cite{MW24} proved this conjecture for $r=3$ by establishing some Rogers--Ramanujan type identities for the corresponding Nahm sums. As the rank increases, evaluating the tadpole Nahm sums becomes much more complicated, which makes the conjecture difficult when $r\geq 4$.

In this paper, our aim is to prove the above conjecture for the next two cases $r=4,5$. A major new ingredient making this possible is that we establish the following rank reduction formulas.
\begin{theorem}\label{thm-rank-reduction}
If $x_{2i}x_{2i+1}=1$ for $i=0,1,\dots,r-1$,  then we have
\begin{align}\label{eq-reduce-even}
    &\chi_{2r}(x_1,\dots,x_{2r};q)=\frac{(-q^{\frac{1}{2}}x_{2r};q)_{\infty}}{(q;q)_{\infty}^{r-1}}\sum_{n_1,\dots,n_{r+1}\geq0}\Bigg(\frac{q^{\frac{1}{2}n_1^2+\frac{1}{2}(n_1+n_2)^2+\frac{1}{4}\sum_{i=2}^r(n_i-n_{i+1})^2}}{(q;q)_{n_1}\cdots (q;q)_{n_{r+1}}} \nonumber\\
    & \quad \times  x_0^{-n_1}\prod_{j=1}^r(\frac{x_{2j}}{x_{2j-2}})^{n_{j+1}} \prod_{k=2}^r\Big(\sum_{\ell_k\in\mathbb{Z}}q^{(\ell_k-\frac{n_k-n_{k+1}}{2})^2}x_{2k-2}^{-\ell_k}\Big)\Bigg).
\end{align}
If $x_{2i-1}x_{2i}=1$ for $i=1,\dots,r$, then we have
    \begin{align}\label{eq-reduce-odd}
        &\chi_{2r+1}(x_1,...,x_{2r+1};q)=\frac{(-q^{\frac{1}{2}}x_{2r+1};q)_{\infty}}{(q;q)_{\infty}^{r}}  \sum_{n_{1},\dots,n_{r}\geq 0}\Bigg(\frac{q^{\frac{1}{4}\sum_{i=1}^{r}(n_i-n_{i-1})^2}}{(q;q)_{n_{1}}\cdots (q;q)_{n_{r}}}\nonumber\\
        &\quad \times\prod_{j=1}^{r}(\frac{x_{2j+1}}{x_{2j-1}})^{n_j}\prod_{k=1}^{r}\Big(\sum_{\ell_{k}\in \mathbb{Z}}q^{(\ell_{k}-\frac{n_{k-1}-n_{k}}{2})^2}x_{2k-1}^{-\ell_k}\Big)\Bigg)
    \end{align}
where $n_0=0$.
\end{theorem}
These formulas reduce the calculations of a rank $2r$ (resp.\ $2r+1$) tadpole Nahm sum to the evaluation of some rank $r+1$ (resp.\ $r$) Nahm sums. In particular, the formula \eqref{eq-reduce-even} transforms the rank two tadpole Nahm sums to some other double Nahm sums, which does not reduce the rank. The formulas \eqref{eq-reduce-even} and \eqref{eq-reduce-odd} will be very helpful when the rank $r\geq 3$. With these formulas we can prove Conjecture \ref{conj-CMP} in a uniform way when $3\leq r\leq 5$. In particular, we are able to prove Conjecture \ref{conj-CMP} for the unsolved cases $r=4,5$ and obtain some companion identities.

To evaluate rank four tadpole Nahm sums, we establish the following new Rogers--Ramanujan type identities.
\begin{theorem}\label{thm-dual-Zagier}
We have
    \begin{align}
        &\sum_{n_1,n_2,n_3\geq0}\frac{q^{4n_1^2+4n_1n_2+3n_2^2-2n_2n_3+n_3^2}}{(q^4;q^4)_{n_1}(q^4;q^4)_{n_2}(q^4;q^4)_{n_3}}=\frac{(q^6;q^6)_{\infty}(q^{12};q^{12})_{\infty}}{(q^3;q^3)_{\infty}(q^8;q^8)_{\infty}(q,q^{11};q^{12})_{\infty}}, \label{dZ1} \\
        &\sum_{n_1,n_2,n_3\geq0}\frac{q^{4n_1^2+4n_1n_2+3n_2^2-2n_2n_3+n_3^2-2n_2+2n_3}}{(q^4;q^4)_{n_1}(q^4;q^4)_{n_2}(q^4;q^4)_{n_3}}=\frac{(q^{12};q^{12})_{\infty}}{(q^8;q^8)_{\infty}(q,q^{11};q^{12})_{\infty}(q^5,q^7;q^{12})_{\infty}},\label{dZ2} \\
        &\sum_{n_1,n_2,n_3\geq0}\frac{q^{4n_1^2+4n_1n_2+3n_2^2-2n_2n_3+n_3^2+2n_3}}{(q^4;q^4)_{n_1}(q^4;q^4)_{n_2}(q^4;q^4)_{n_3}}=\frac{(q^6;q^6)_{\infty}^5(q^8;q^8)_{\infty}}{(q^3;q^3)_{\infty}^2(q^4;q^4)_{\infty}^2(q^{12};q^{12})_{\infty}^2},\label{dZ3} \\
        &\sum_{n_1,n_2,n_3\geq0}\frac{q^{4n_1^2+4n_1n_2+3n_2^2-2n_2n_3+n_3^2+4n_1+2n_2+2n_3}}{(q^4;q^4)_{n_1}(q^4;q^4)_{n_2}(q^4;q^4)_{n_3}}=\frac{(q^6;q^6)_{\infty}(q^{12};q^{12})_{\infty}}{(q^3;q^3)_{\infty}(q^8;q^{8})_{\infty}(q^5,q^7;q^{12})_{\infty}}. \label{dZ4}
\end{align}
\end{theorem}

We also find a companion identity which we are not able to prove at this stage, and we leave it as an open problem.
\begin{conj}\label{conj-dZ5}
We have
\begin{align}
 &\sum_{n_1,n_2,n_3\geq0}\frac{q^{4n_1^2+4n_1n_2+3n_2^2-2n_2n_3+n_3^2+4n_1+2n_2}}{(q^4;q^4)_{n_1}(q^4;q^4)_{n_2}(q^4;q^4)_{n_3}}=\frac{(q^6;q^6)_{\infty}(q^8;q^8)_{\infty}(q^2,q^{10};q^{12})_{\infty}}{(q^4;q^4)_{\infty}^2(q,q^{11};q^{12})_{\infty}(q^5,q^7;q^{12})_{\infty}}. \label{dZ5}
 \end{align}
\end{conj}
Note that the matrix and vector parts for the Nahm sums involved in \eqref{dZ1}--\eqref{dZ5} are
\begin{align}
    A=\begin{pmatrix}
    2&1&0\\1& {3}/{2}& -{1}/{2}\\ 0&-{1}/{2}& {1}/{2}
\end{pmatrix}, \quad B\in \left\{\begin{pmatrix}
    0 \\ 0 \\ 0
\end{pmatrix}, \begin{pmatrix} 0 \\ -{1}/{2} \\ {1}/{2} \end{pmatrix}, \begin{pmatrix} 0 \\ 0 \\ 1/2 \end{pmatrix}, \begin{pmatrix} 1 \\ 1/2 \\ 1/2 \end{pmatrix}, \begin{pmatrix} 1 \\ 1/2 \\ 0 \end{pmatrix} \right\}.
\end{align}
This is dual in the sense of \eqref{eq-dual} to Zagier's eleventh rank three example \cite[Table 3]{Zagier}:
\begin{align}
A^{-1}=\begin{pmatrix}
    1&-1&-1\\-1&2&2\\-1&2&4
\end{pmatrix}, \quad B\in \left\{\begin{pmatrix}
    0 \\ 0 \\ 0
\end{pmatrix}, \begin{pmatrix} 0 \\ 0 \\ 1 \end{pmatrix}, \begin{pmatrix} -1/2 \\ 1 \\ 2 \end{pmatrix}, \begin{pmatrix} 0 \\ 1 \\ 2 \end{pmatrix} , \begin{pmatrix}
    1/2 \\ 0 \\ 0
\end{pmatrix} \right\}.
\end{align}
Its modularity was proved by Wang \cite[Theorem 4.14]{Wang-rank3}. Theorem \ref{thm-dual-Zagier} together with Conjecture \ref{conj-dZ5} confirms the modularity of its dual triples for the first time.

Based on Theorem \ref{thm-dual-Zagier} and Conjecture \ref{conj-dZ5}, we prove five modular cases of generalized tadpole Nahm sums in the rank four case. To make our formulas more compact, we use the following notation:
\begin{align}
    J_m:=(q^m;q^m)_{\infty},\quad J_{a,m}:=(q^a,q^{m-a},q^m;q^m)_{\infty}.
\end{align}
\begin{theorem}\label{thm-rank4}
We have
\begin{align}
        \chi_4(1,1,1,1;q^2)&=\frac{(-q;q^2)_{\infty}}{(q^2;q^2)_{\infty}}(\frac{J_4^5J_{12}J_{4,24}J_{10,24}}{J_1J_2^2J_8^2J_{24}^2}+2q\frac{J_8^2J_{12}J_{5,12}J_{2,24}^2J_{10,24}}{J_1J_4J_{1,12}J_{4,24}J_{24}^2}), \label{eq-4-1}\\
        \chi_4(1,q^{-2},q^2,1;q^2)&=\frac{(-q;q^2)_{\infty}}{(q^2;q^2)_{\infty}}\left(\frac{J_4^5J_{24}^3}{J_1J_2J_8^3J_{2,24}J_{10,24}}+2\frac{J_2J_8J_{24}^3}{J_1J_4J_{4,24}^2}\right), \label{eq-4-2}\\
        \chi_4(1,1,1,q;q^2)&=\frac{(-q^2;q^2)_{\infty}}{(q^2;q^2)_{\infty}}\left(\frac{J_4^5J_{12}^6}{J_2^3J_6^3J_8^2J_{2,24}J_{10,24}}+4q^2\frac{J_8^2J_{24}^4}{J_2J_4J_6J_{2,24}J_{10,24}}\right),  \label{eq-4-3} \\
\chi_4(q^2,q^{-2},q^2,1;q^2)&=\frac{(-q;q^2)_{\infty}}{(q^2;q^2)_{\infty}}\left(q\frac{J_4^5J_{24}^3}{J_1J_2J_8^3J_{6,24}J_{10,24}}+2\frac{J_8^3J_{12}^3J_{24}^2J_{10,24}}{J_3J_4^3J_{5,24}^2J_{7,24}^2}\right). \label{eq-4-4}
\end{align}
Assume Conjecture \ref{conj-dZ5} we also have
\begin{align}
\chi_4(q^2,q^{-2},q^2,q^{-1};q^2)&=\frac{(-1;q^2)_{\infty}}{(q^2;q^2)_{\infty}}\left(\frac{J_4^5J_{12}^2J_{24}^7}{J_2^2J_6^2J_8^3J_{2,24}^3J_{10,24}^3}+2\frac{J_8^3J_{12}^2}{J_2^2J_4J_{24}}\right). \label{eq-4-5}
\end{align}
\end{theorem}

Meanwhile, we find nine identities for the generalized tadpole Nahm sums in the rank five case (see Theorem \ref{thm-rank5}). For instance, we prove that
\begin{align}\label{intro-5-1}
   \chi_5(1,1,1,1,1;q^4)&=4q^2\frac{J_8^6J_{8,28}J_{12,56}}{J_2J_4^5J_8J_{56}}-2q\frac{J_4J_8J_{6,28}J_{16,56}}{J_2^3J_{56}}\nonumber\\&+\frac{1}{2}\frac{J_2^7J_{6,28}J_{16,56}}{J_1^4J_4J_8^3J_{56}}+\frac{1}{2}\frac{J_2^5J_{28}^2J_{12,56}J_{16,56}}{J_4^5J_8J_{56}^2J_{6,28}}.
\end{align}
As a consequence of \eqref{eq-4-1} and \eqref{intro-5-1}, we immediately obtain the following
\begin{corollary}
Conjecture \ref{conj-CMP} holds when $r=4,5$.
\end{corollary}

The rest of this paper is organized as follows. In Section \ref{sec-pre} we collect some auxiliary results. In particular, we discuss convergence conditions for some infinite series under which we can interchange the order of summation. We establish the rank reduction formulas \eqref{eq-reduce-even} and \eqref{eq-reduce-odd} in Section \ref{sec-rank-reduction}. We prove Theorems \ref{thm-dual-Zagier} and \ref{thm-rank4} in Section \ref{sec-rank4}. Modular representations for the rank five tadpole Nahm sums will be given in Section \ref{sec-rank5}. Finally, we give an unified new proof for the case of rank three in Section \ref{sec-remark} and point out the difficulty for higher rank cases.

\section{Preliminaries}\label{sec-pre}
The $q$-binomial theorem \cite[Theorem 2.1]{Andrews-book} asserts that
\begin{align}
    \sum_{n=0}^{\infty}\frac{(a;q)_n}{(q;q)_n}z^n=\frac{(az;q)_{\infty}}{(z;q)_{\infty}},\quad \abs{z}<1.\label{qbi}
\end{align}
As consequences, we have Euler's $q$-exponential identities \cite[Corollary 2.2]{Andrews-book} given by \eqref{euler-2} and
\begin{align}
    \sum_{n=0}^{\infty} \frac{z^{n}}{(q ; q)_{n}}=\frac{1}{(z ; q)_{\infty}}, \quad |z|<1. \label{euler}
\end{align}
The $q$-binomial coefficient or Gaussian coefficient is defined as
$${n\brack m}={n \brack m}_q:=\left\{\begin{array}{ll}
\frac{(q;q)_n}{(q;q)_m (q;q)_{n-m}}, & 0\leq m \leq n, \\
0, & \text{otherwise}. \end{array}\right.$$
As a finite version of \eqref{euler-2}, we have \cite[Theorem 3.3]{Andrews-book}
\begin{align}\label{eq-finite}
(-z;q)_n=\sum_{i=0}^n { n\brack i}z^i q^{i(i-1)/2}.
\end{align}

Recall the Jacobi triple product identity \cite[Theorem 2.8]{Andrews-book}
\begin{align}\label{JTP}
\theta(z):=\sum_{n=-\infty}^\infty q^{\frac{1}{2}n^2}z^{n}=(-q^{1/2}z,-q^{1/2}/z,q;q)_\infty.
\end{align}
As special instances, we have Ramanujan's functions \cite[Corollary 2.10]{Andrews-book}
\begin{align}
    \varphi(q):=\sum_{n=-\infty}^\infty q^{n^2}=\frac{J_2^5}{J_1^2J_4^2}, \quad
    \psi(q):=\sum_{n=0}^\infty q^{n(n+1)/2}=\frac{J_2^2}{J_1}. \label{eq-psi}
\end{align}
It is easy to check that
\begin{align}
   \theta_0=\theta_0(q)&:=\sum_{i=-\infty}^\infty q^{(2i)^2}=\frac{1}{2}\left(\varphi(q)+\varphi(-q)\right)=\varphi(q^4)=\frac{J_8^5}{J_4^2J_{16}^2}, \label{theta0-defn} \\
\theta_1=\theta_1(q)&:=\sum_{i=-\infty}^\infty q^{(2i+1)^2}= \frac{1}{2}\left(\varphi(q)-\varphi(-q)\right)= 2q\psi(q^8)=2q\frac{J_{16}^2}{J_8}. \label{theta1-defn}
\end{align}
We also need Heine's transformation formula \cite[Corollary 2.4]{Andrews-book}:
\begin{align}
    \sum_{n=0}^{\infty}\frac{(a,b;q)_nt^n}{(c,q;q)_n}=\frac{(b,at;q)_{\infty}}{(c,t;q)_{\infty}}\sum_{n=0}^{\infty}\frac{(c/b,t;q)_nb^n}{(at,q;q)_n}.\label{Heine}
\end{align}

We will frequently use the fact: for $n\geq 0$,
\begin{align}\label{aq-finite}
    (aq^{-n};q)_\infty=(-a)^nq^{-n(n+1)/2}(a;q)_\infty (q/a;q)_n.
\end{align}
As some special instances of \eqref{aq-finite}, we have
\begin{align}
&(q^{-i};q)_\infty=0, \quad i\geq 0, \label{finite-1}\\
&(q^{-i-1/2};q)_\infty=(-1)^{i+1}q^{-(i+1)^2/2}(q^{1/2};q)_\infty (q^{1/2};q)_{i+1}, \label{finite-2}\\
&(q^{1/2-i};q)_\infty=(-1)^iq^{-i^2/2}(q^{1/2};q)_\infty(q^{1/2};q)_i, \label{finite-add} \\
    &(-q^{-i};q)_{\infty}=2q^{-(i^2+i)/2}(-q;q)_{\infty}(-q;q)_i, \label{finite-3}\\
    &(-q^{-i-1/2};q)_{\infty}=q^{-(i+1)^2/2}(-q^{1/2};q)_{\infty}(-q^{1/2};q)_{i+1}, \label{finite-4}\\
    &(-q^{1-i};q)_{\infty}=\begin{cases}(-q;q)_{\infty}& \text{if}\ i=0\\2q^{-(i^2-i)/2}(-q;q)_{\infty}(-q;q)_{i-1}& \text{if} \ i\geq1 \\\end{cases}, \label{finite-5}\\
    &(-q^{1/2-i};q)_{\infty}=q^{-i^2/2}(-q^{1/2};q)_{\infty}(-q^{1/2};q)_i. \label{finite-6}
\end{align}

Replacing $z$ by $zq^{(1-n)/2}$ in \eqref{eq-finite}, we obtain
    \begin{align}\label{eq-lemma-sum}
        \sum_{\substack{i,k\geq0 \\ i+k=n}}\frac{q^{-\frac{ik}{2}}z^k}{(q;q)_i(q;q)_k}=\frac{1}{(q;q)_n}\cdot\frac{(-zq^{\frac{1-n}{2}};q)_{\infty}}{(-zq^{\frac{1+n}{2}};q)_{\infty}}.
    \end{align}
Using \eqref{eq-lemma-sum} with some specific $z$, we obtain the following identities which will be used in the proof of Theorem \ref{thm-dual-Zagier}:
\begin{align}
    \sum_{\substack{n,r\geq0\\n+r=2i}}\frac{q^{-\frac{nr}{2}+\frac{n}{2}}(-1)^n}{(q;q)_n(q;q)_r}&= \begin{cases}1& \text{if}\ i=0 \\0& \text{if}\ i\geq1  \end{cases}, \label{finite-sum-1}\\
    \sum_{\substack{n,r\geq0\\n+r=2i+1}}\frac{q^{-\frac{nr}{2}+\frac{n}{2}}(-1)^n}{(q;q)_n(q;q)_r}&=\frac{(-1)^iq^{-\frac{i^2}{2}}(q^{\frac{1}{2}};q)_i(q^{\frac{1}{2}};q)_{i+1}}{(q;q)_{2i+1}}, \label{finite-sum-2}\\
    \sum_{\substack{n,r\geq0\\n+r=2i}}\frac{q^{-\frac{nr}{2}+\frac{n}{2}}}{(q;q)_n(q;q)_r}&=q^{-\frac{i^2-i}{2}}\frac{(-1;q)_i(-q;q)_i}{(q;q)_{2i}}, \label{finite-sum-3}\\
    \sum_{\substack{n,r\geq0\\n+r=2i+1}}\frac{q^{-\frac{nr}{2}+\frac{n}{2}}}{(q;q)_n(q;q)_r}&=q^{-\frac{i^2}{2}}\frac{(-q^{\frac{1}{2}};q)_i(-q^{\frac{1}{2}};q)_{i+1}}{(q;q)_{2i+1}}, \label{finite-sum-4}\\
    \sum_{\substack{n,r\geq0\\n+r=2i}}\frac{q^{-\frac{nr}{2}}(-1)^n}{(q;q)_n(q;q)_r}&=(-1)^iq^{-\frac{i^2}{2}}\frac{(q^{\frac{1}{2}};q)_i^2}{(q;q)_{2i}},  \label{finite-sum-5}\\
    \sum_{\substack{n,r\geq0\\n+r=2i+1}}\frac{q^{-\frac{nr}{2}}(-1)^n}{(q;q)_n(q;q)_r}&=0,  \label{finite-sum-6} \\
     \sum_{\substack{n,r\geq0\\n+r=2i}}\frac{q^{-\frac{nr}{2}}}{(q;q)_n(q;q)_r}&=q^{-\frac{i^2}{2}}\frac{(-q^{\frac{1}{2}};q)_i^2}{(q;q)_{2i}},  \label{finite-sum-6-add}\\
    \sum_{\substack{n,r\geq0\\n+r=2i+1}}\frac{q^{-\frac{nr}{2}}}{(q;q)_n(q;q)_r}&=2q^{-\frac{i(i+1)}{2}}\frac{(-q;q)_i^2}{(q;q)_{2i+1}},  \label{finite-sum-7} \\
    \sum_{\substack{n,r\geq0\\n+r=2i}}\frac{q^{-\frac{nr}{2}-\frac{n}{2}+\frac{r}{2}}(-1)^n}{(q;q)_n(q;q)_r}&= \begin{cases}1& \text{if}\ i=0 \\ (-1)^{i+1}q^{-\frac{i^2+1}{2}}\frac{(q^{\frac{1}{2}};q)_{i-1}(q^{\frac{1}{2}};q)_{i+1}}{(q;q)_{2i}}&
\text{if} \ i\geq 1\end{cases},  \label{finite-sum-8} \\
    \sum_{\substack{n,r\geq0\\n+r=2i}}\frac{q^{-\frac{nr}{2}-\frac{n}{2}}(-1)^n}{(q;q)_n(q;q)_r}&= \begin{cases}1& \text{if}\ i=0 \\ 0& \text{if}\ i\geq 1 \\ \end{cases}, \label{finite-sum-9} \\
    \sum_{\substack{n,r\geq0\\n+r=2i+1}}\frac{q^{-\frac{nr}{2}-\frac{n}{2}}(-1)^n}{(q;q)_n(q;q)_r}&=(-1)^{i+1}q^{-\frac{(i+1)^2}{2}}\frac{(q^{\frac{1}{2}};q)_i(q^{\frac{1}{2}};q)_{i+1}}{(q;q)_{2i+1}}, \label{finite-sum-10} \\
    \sum_{\substack{n,r\geq0\\n+r=2i+1}}\frac{q^{-\frac{nr}{2}-\frac{n}{2}}}{(q;q)_n(q;q)_r}&=q^{-\frac{(i+1)^2}{2}}\frac{(-q^{\frac{1}{2}};q)_i(-q^{\frac{1}{2}};q)_{i+1}}{(q;q)_{2i+1}}. \label{finite-sum-11}
\end{align}

Recall the following known Rogers--Ramanujan type identities:
\begin{align}
 &\sum_{n=0}^\infty \frac{(-q;q^2)_nq^{n^2}}{(q^4;q^4)_n}=\frac{J_2J_3^2}{J_1J_4J_6}, \quad\text{({S.\ 25})}\label{S.25}\\
  &\sum_{n=0}^\infty\frac{(-q;q^2)_nq^{n^2+2n}}{(q^4;q^4)_n}=\frac{J_6^2}{J_3J_4},\quad   \text{(Ramanujan \cite[Entry 4.2.11]{LostNotebook2})}\label{Ramanujan[10 Entry4.2.11]}\\
     &\sum_{n=0}^\infty \frac{(-q^2;q^2)_nq^{n(n+1)}}{(q;q)_{2n+1}}=\frac{J_3J_{12}}{J_1J_6}, \quad  \text{(S.\ 28)}\label{S.28} \\
    &\sum_{n=0}^\infty \frac{(-q;q^2)_nq^{n^2}}{(q;q)_{2n}}=\frac{J_6^2}{J_1J_{12}}, \quad \text{({S.\ 29})}\label{S.29}\\
    &\sum_{n=0}^\infty \frac{(-q;q^2)_nq^{n^2+2n}}{(q;q)_{2n+1}}=\frac{J_2J_{12}^2}{J_1J_4J_6}, \quad\text{({S.\ 50})}\label{S.50}\\
    &\sum_{n=0}^\infty \frac{(-q^2;q^2)_nq^{n^2+n}}{(q^2;q^2)_{2n+1}}=\frac{J_4J_{4,14}J_{6,28}}{J_2^2J_{28}}, \quad \text{(S.\ 80)}\label{S.80}\\
    &\sum_{n=0}^\infty \frac{(-q^2;q^2)_nq^{n^2+n}}{(q^2;q^2)_{2n}}=\frac{J_4J_{2,14}J_{10,28}}{J_2^2J_{28}},  \quad \text{(S.\ 81)}\label{S.81}\\
    &\sum_{n=0}^\infty \frac{(-q^2;q^2)_nq^{n^2+3n}}{(q^2;q^2)_{2n+1}}=\frac{J_4J_{6,14}J_{2,28}}{J_2^2J_{28}}, \quad \text{(S.\ 82)}\label{S.82}\\
    &\sum_{n=0}^\infty \frac{(-q;q^2)_nq^{n^2}}{(q^2;q^2)_{2n}}=\frac{J_2J_{3,14}J_{8,28}}{J_1J_4J_{28}}, \quad \text{(S.\ 117)}\label{S.117}\\
    &\sum_{n=0}^\infty \frac{(-q;q^2)_nq^{n^2+2n}}{(q^2;q^2)_{2n}}=\frac{J_2J_{1,14}J_{12,28}}{J_1J_4J_{28}}, \quad  \text{(S.\ 118)}\label{S.118}\\
    &\sum_{n=0}^\infty \frac{(-q;q^2)_{n+1}q^{n^2+2n}}{(q^2;q^2)_{2n+1}}=\frac{J_2J_{5,14}J_{4,28}}{J_1J_4J_{28}}. \quad \text{(S.\ 119)}\label{S.119}
\end{align}
Most of them can be found in Slater's list \cite{Slater} and we use (S.\ $n$) to denote the $n$-th identity in \cite{Slater}.

We define the constant term extractor $\CT$ as
\begin{align}
{\CT}_{z_1,z_2,\dots,z_k}\Big(\sum_{n_1,n_2,\dots,n_k\in\mathbb{Z}}a_{n_1n_2\dots n_k}z_1^{n_1}z_2^{n_2}\cdots z_k^{n_k}\Big):=a_{00\cdots0}.
\end{align}
This operator extracts the constant term with respect to the variables $z_1,z_2,\dots,z_k$.

We now discuss some conditions to guarantee convergence.
\begin{lemma}\label{lem-convergence}
Let $m$ be a positive integer and $c>0$, $t>1$. The series $\sum_{n=0}^\infty \frac{z^n}{(q;q)_n^m}$ converges absolutely when $|z|<1$, and the series $\sum_{n=0}^\infty \frac{q^{cn^t}z^n}{(q;q)_n^m}$ converges absolutely for any $z\in \mathbb{C}$.
\end{lemma}
\begin{proof}
Note that
\begin{align}
\lim_{n\rightarrow \infty}\frac{|z^n(q;q)_{n}^{-m}|}{|z^{n-1}(q;q)_{n-1}^{-m}|}=\lim_{n\rightarrow \infty}\frac{|z|}{|(1-q^n)^m|}=|z|.
\end{align}
By the ratio test, we get the first assertion.

For $|a|<1$ we have
\begin{align}
    &|(a;q)_n|=|(1-a)(1-aq)\cdots (1-aq^{n-1})| \nonumber \\
    &\geq (1-|a|)(1-|a||q|)\cdots (1-|a||q|^{n-1})=(|a|;|q|)_n. \label{ineq-1}
\end{align}

Note that for any $z\in \mathbb{C}$, there exists some positive $\epsilon$ such that $|q|^\epsilon |z|<1$. There exists some positive integer $n_0$ such that $cn^t\geq \epsilon n$ whenever $n\geq n_0$. Therefore, using \eqref{ineq-1} we deduce that
\begin{align}
    \sum_{n\geq n_0} \left|\frac{q^{cn^t}z^n}{(q;q)_n^m}\right| \leq \sum_{n\geq n_0} \frac{|q|^{\epsilon n}|z|^n}{(|q|;|q|)_n^m}.
\end{align}
Since $|q|^\epsilon |z|<1$, we get the second assertion from the first assertion.
\end{proof}

\begin{lemma}\label{lem-convergence-double}
    Let $c>0$, $t>1$, $c_1,c_2\in \mathbb{R}$ and
\begin{align}
    s_{m,r}(q):=\frac{q^{c|m-r|^t+c_1m+c_2r}z^m}{(q;q)_m(q;q)_r}.
\end{align}
Then $\sum_{m,r\geq 0} s_{m,r}(q)$ converge absolutely when $|z|<|q|^{-c_1-c_2}$.
\end{lemma}
\begin{proof}
From \eqref{ineq-1} we deduce that
\begin{align}
    |(q;q)_{m+r}|&=|(q;q)_r(q^{r+1};q)_m|\geq (|q|;|q|)_r(|q|^{r+1};|q|)_m \nonumber \\
    &\geq (|q|;|q|)_r (|q|;|q|)_m. \label{ineq-2}
\end{align}
We have
\begin{align*}
&\sum_{m,r\geq 0}\left| s_{m,r}(q)\right|=\sum_{m\geq r\geq 0}\left| s_{m,r}(q)\right|+\sum_{0\leq m< r}\left| s_{m,r}(q)\right| \nonumber \\
&=\sum_{m,r\geq 0} \left|\frac{q^{cm^t+c_1(m+r)+c_2r}z^{m+r}}{(q;q)_{m+r}(q;q)_r}\right|
+\sum_{m\geq 0,r>0} \left|\frac{q^{cr^t+c_1m+c_2(m+r)}z^{m}}{(q;q)_{m}(q;q)_{m+r}}\right| \nonumber \\
&\leq \sum_{m,r\geq 0} \frac{|q|^{cm^t+c_1m+(c_1+c_2)r}|z|^{m+r}}{(|q|;|q|)_r^2(|q|;|q|)_m} +\sum_{m\geq 0,r>0} \frac{|q|^{cr^t+(c_1+c_2)m+c_2r}|z|^{m}}{(|q|;|q|)_{m}^2(|q|;|q|)_{r}}  \nonumber \\
&\leq \sum_{m\geq 0} \frac{|q|^{cm^t+c_1m}|z|^m}{(|q|;|q|)_m} \times \sum_{r\geq 0} \frac{|q|^{(c_1+c_2)r}|z|^r}{(|q|;|q|)_r^2}+\sum_{m\geq 0} \frac{|q|^{(c_1+c_2)m}|z|^m}{(|q|;|q|)_m^2} \times \sum_{r>0} \frac{|q|^{cr^t+c_2r}}{(|q|;|q|)_r}. 
\end{align*}
From Lemma \ref{lem-convergence} we see that the series $\sum_{m,r\geq 0} s_{m,r}(q)$ converges absolutely when $|z|<|q|^{-c_1-c_2}$.
\end{proof}

\section{Proof of the rank reduction formulas}\label{sec-rank-reduction}
For convenience, we denote
\begin{align}
    &X_m:=\frac{1}{2}n^\mathrm{T} T_m n=n_1^2+\cdots+n_{m-1}^2+\frac{1}{2}n_{m}^2-n_{1}n_{2}-\cdots-n_{m-1}n_{m}.
\end{align}
Note that $X_1=\frac{1}{2}n_1^2$ and for $m\geq 1$ we have
\begin{align}
    X_m&=X_{m-1}+\frac{1}{2}(n_m-n_{m-1})^2,  \label{X-rec} \\
    X_m&=\frac{1}{2}n_1^2+\frac{1}{2}\sum_{i=1}^{m-1}(n_i-n_{i+1})^2. \label{X-square}
\end{align}

\begin{proof}[Proof of Theorem \ref{thm-rank-reduction}]
Summing over $n_{m}$ first using \eqref{euler-2}, we have
\begin{align}
    &\chi_{m}(x_1,\dots,x_{m};q) =\sum_{n_1,\dots,n_{m-1}\geq 0}\frac{q^{X_{m-1}}x_1^{n_1}\cdots x_{m-1}^{n_{m-1}}(-q^{\frac{1}{2}-n_{m-1}}x_{m};q)_{\infty}}{(q;q)_{n_1}\cdots(q;q)_{n_{m-1}}}\nonumber\\
    &=(-q^{\frac{1}{2}}x_{m};q)_{\infty}\sum_{n_1,\dots,n_{m-1}\geq 0}\frac{q^{X_{m-1}}x_1^{n_1}\cdots x_{m-1}^{n_{m-1}}x_{m}^{n_{m-1}}(-{q^{\frac{1}{2}}}/{x_{m}};q)_{n_{m-1}}}{(q;q)_{n_1}\cdots(q;q)_{n_{m-1}}}\nonumber\\
    &=(-q^{\frac{1}{2}}x_{m};q)_{\infty}\nonumber\\
    & \quad \times\CT_{z_1}\Bigg[ \sum_{\substack{n_1,\dots,n_{m-1}\geq 0\\s_{m-2}\in\mathbb{Z}}}\frac{q^{X_{m-2}+\frac{1}{2}s_{m-2}^2}x_1^{n_1}\cdots x_{m-1}^{n_{m-1}}x_{m}^{n_{m-1}}z_1^{n_{m-2}-n_{m-1}-s_{m-2}}(-{q^{\frac{1}{2}}}/{x_{m}};q)_{n_{m-1}}}{(q;q)_{n_1}\cdots(q;q)_{n_{m-1}}}\Bigg]\nonumber\\
    &=(-q^{\frac{1}{2}}z_{m};q)_{\infty}\CT_{z_1}\Bigg[\sum_{n_1,\dots,n_{m-2}\geq 0}\frac{q^{X_{m-2}}x_1^{n_1}\cdots x_{m-2}^{n_{m-2}}z_1^{n_{m-2}}}{(q;q)_{n_1}\dots(q;q)_{n_{m-2}}}\times \sum_{s_{m-2}\in\mathbb{Z}}q^{\frac{1}{2}s_{m-2}^2}z_1^{-s_{m-2}} \nonumber\\
    & \quad \times  \sum_{n_{m-1}\geq 0}\frac{(\frac{-q^{\frac{1}{2}}}{x_{m}};q)_{n_{m-1}} (\frac{x_{m-1}x_{m}}{z_1})^{n_{m-1}}}{(q;q)_{n_{m-1}}}\Bigg]\nonumber\\
    &=(-q^{\frac{1}{2}}x_{m};q)_{\infty}\CT_{z_1}\Bigg[\sum_{n_1,\dots,n_{m-2}\geq 0}\frac{q^{X_{m-2}}x_1^{n_1}\cdots x_{m-2}^{n_{m-2}}z_1^{n_{m-2}}}{(q;q)_{n_1}\dots(q;q)_{n_{m-2}}}\times \theta(z_1)\times \frac{(\frac{-q^{\frac{1}{2}}x_{m-1}}{z_1};q)_{\infty}}{(\frac{x_{m-1}x_{m}}{z_1};q)_{\infty}}\Bigg].
    \end{align}
Here for the second equality we used \eqref{aq-finite} and for the last equality we used \eqref{qbi}.

In the multi sums above we sum over $n_{m-2}$ first using \eqref{euler}, and arguing similarly as before we deduce that
\begin{align}
    & \chi_{m}(x_1,\dots,x_{m};q) \nonumber \\
    &=(-q^{\frac{1}{2}}x_{m};q)_{\infty}\CT_{z_1}\Bigg[\sum_{n_1,\dots,n_{m-3}\geq 0}\frac{q^{X_{m-3}+\frac{1}{2}n_{m-3}^2}x_1^{n_1}\cdots x_{m-3}^{n_{m-3}}(-q^{\frac{1}{2}-n_{m-3}}z_1x_{m-2};q)_\infty}{(q;q)_{n_1}\dots(q;q)_{n_{m-3}}} \nonumber\\
    &\quad \times \theta(z_1) \frac{(\frac{-q^{\frac{1}{2}}x_{m-1}}{z_1};q)_{\infty}}{(\frac{x_{m-1}x_{m}}{z_1};q)_{\infty}}\Bigg] \nonumber \\
    &=(-q^{\frac{1}{2}}x_{m};q)_{\infty}\CT_{z_1}\Bigg[\sum_{n_1,\dots,n_{m-3}\geq 0}\frac{q^{X_{m-3}}x_1^{n_1}\cdots x_{m-3}^{n_{m-3}}x_{m-2}^{n_{m-3}}z_1^{n_{m-3}}}{(q;q)_{n_1}\dots(q;q)_{n_{m-3}}}\times \theta(z_1)  \nonumber\\
    &\quad \times  \frac{(\frac{-q^{\frac{1}{2}}x_{m-1}}{z_1};q)_{\infty}(-q^{\frac{1}{2}}z_1x_{m-2};q)_{\infty}(\frac{-q^{\frac{1}{2}}}{z_1x_{m-2}};q)_{n_{m-3}}}{(\frac{x_{m-1}x_{m}}{z_1};q)_{\infty}}\Bigg] \nonumber \\
     &=\frac{(-q^{\frac{1}{2}}x_{m};q)_{\infty}}{(q;q)_{\infty}} \nonumber\\
    &\quad \times \CT_{z_1}\Bigg[\sum_{n_1,...,n_{m-3}\geq 0}\frac{q^{X_{m-3}}x_1^{n_1}\cdots x_{m-3}^{n_{m-3}}x_{m-2}^{n_{m-3}}z_1^{n_{m-3}}(\frac{-q^{\frac{1}{2}}}{z_1x_{m-2}};q)_{n_{m-3}}}{(q;q)_{n_1}\dots(q;q)_{n_{m-3}}}\times \frac{\theta(z_1)\theta(\frac{z_1}{x_{m-1}})}{(\frac{x_{m-1}x_{m}}{z_1};q)_{\infty}}\Bigg].
\end{align}
Here for the last equality we used the fact $x_{m-2}x_{m-1}=1$.

Repeating the above two steps and introducing a new variable $z_2$, we deduce that
\begin{align}
        & \chi_{m}(x_1,...,x_{m};q) \nonumber \\
        &=\frac{(-q^{\frac{1}{2}}x_{m};q)_{\infty}}{(q;q)_{\infty}^2}\CT_{z_1,z_2}\Bigg[\sum_{n_1,\dots,n_{m-5}\geq 0}\frac{q^{X_{m-5}}x_1^{n_1}\cdots x_{m-5}^{n_{m-5}}x_{m-4}^{n_{m-5}}z_2^{n_{m-5}}(\frac{-q^{\frac{1}{2}}}{z_2x_{m-4}};q)_{n_{m-5}}}{(q;q)_{n_1}\dots(q;q)_{n_{m-5}}} \nonumber\\
    & \quad \times \frac{\theta(z_1)\theta(z_2)\theta(\frac{z_1}{x_{m-1}})\theta(\frac{z_2}{x_{m-3}})}{(\frac{x_{m-1}x_{m}}{z_1},\frac{x_{m-3}x_{m-2}z_1}{z_2};q)_{\infty}}\Bigg].
\end{align}
Continuing this process and introducing new variables $z_3,z_4,\dots$, we express the generalized tadpole Nahm sums as constant terms of some infinite products. Specifically speaking, when $m=2r$ is even, we deduce that
\begin{align}
    &\chi_{2r}(x_1,\dots,x_{2r};q) \nonumber \\
    &=\frac{(-q^{\frac{1}{2}}x_{2r};q)_{\infty}}{(q;q)_{\infty}^{r-1}}\CT_{z_1,\dots,z_{r-1}}\Bigg[\sum_{n_1\geq0}\frac{q^{\frac{1}{2}n_1^2}(z_{r-1}x_1x_2)^{n_1}(\frac{-q^{\frac{1}{2}}}{z_{r-1}x_{2}};q)_{n_1}}{(q;q)_{n_1}} \nonumber\\
    &\quad \times  \frac{\theta(z_1)\theta(z_2)\cdots \theta(z_{r-1})\theta(z_1x_{2r-2})\theta(z_2x_{2r-4})\cdots \theta(z_{r-1}x_2)}{(\frac{x_{2r}}{z_1x_{2r-2}},\frac{z_1x_{2r-2}}{z_2x_{2r-4}},\dots,\frac{z_{r-2}x_4}{z_{r-1}x_2};q)_{\infty}}\Bigg]\nonumber\\
    &=\frac{(-q^{\frac{1}{2}}x_{2r};q)_{\infty}}{(q;q)_{\infty}^{r-1}}\CT_{z_1,\dots,z_{r}}\Bigg[\frac{\theta(z_1)\cdots \theta(z_r)\theta(z_1x_{2r-2})\cdots \theta(z_{r-1}x_2)(\frac{-q^{\frac{1}{2}}}{z_r x_0};q)_{\infty}}{(\frac{x_{2r}}{z_1x_{2r-2}},\dots,\frac{z_{r-2}x_4}{z_{r-1}x_2},\frac{z_{r-1}x_2}{z_r x_0};q)_{\infty}}\Bigg]\nonumber\\
    &=\frac{(-q^{\frac{1}{2}}x_{2r};q)_{\infty}}{(q;q)_{\infty}^{r-1}}\CT_{z_1,\dots,z_{r}}\Bigg[\frac{\theta(z_1)\cdots\theta(z_r)\theta(z_2x_{2})\cdots\theta(z_{r}x_{2r-2})(\frac{-q^{\frac{1}{2}}}{z_1x_0};q)_{\infty}}{(\frac{x_{2r}}{z_rx_{2r-2}},\dots,\frac{z_{3}x_4}{z_{2}x_2},\frac{z_{2}x_2}{z_1x_0};q)_{\infty}}\Bigg].\label{proof-even-constant}
    \end{align}
Here for the last equality we interchange the symbol $z_j$ with $z_{r+1-j}$ for $j=1,2,\dots,r$, respectively.

When $m=2r+1$ is odd we deduce that
\begin{align}
        & \chi_{2r+1}(x_1,\dots,x_{2r+1};q) \nonumber \\
        &=\frac{(-q^{\frac{1}{2}}x_{2r+1};q)_{\infty}}{(q;q)_{\infty}^{r-1}}\CT_{z_1,\dots,z_{r-1}}\left[\sum_{n_1,n_2\geq 0}\frac{q^{X_2}x_1^{n_1}x_2^{n_2}x_{3}^{n_2}z_{r-1}^{n_2}(\frac{-q^{\frac{1}{2}}}{z_{r-1}x_3};q)_{n_2}}{(q;q)_{n_1}(q;q)_{n_2}} \right.\nonumber\\
    &\quad \times \left.\frac{\theta(z_1)\theta(z_2)\cdots\theta(z_{r-1})\theta(z_1x_{2r-1})\theta(z_2x_{2r-3})\cdots\theta(z_{r-1}x_3)}{(\frac{x_{2r+1}}{z_1x_{2r-1}},\frac{z_1x_{2r-1}}{z_2x_{2r-3}},...,\frac{z_{r-2}x_{5}}{z_{r-1}x_{3}};q)_{\infty}}\right]\nonumber\\
    &=\frac{(-q^{\frac{1}{2}}x_{2r+1};q)_{\infty}}{(q;q)_{\infty}^{r-1}}\CT_{z_1,\dots,z_r}\left[\sum_{n_1\geq 0}\frac{q^{\frac{1}{2}n_1^2}x_1^{n_1}z_r^{n_1}}{(q;q)_{n_1}}\times \sum_{n_2\geq0}\frac{(\frac{-q^{\frac{1}{2}}}{z_{r-1}x_3};q)_{n_2}x_2^{n_2}x_3^{n_2}z_{r-1}^{n_2}z_r^{-n_2}}{(q;q)_{n_2}} \right.\nonumber\\
    &\quad \times \left.\sum_{s_1\in\mathbb{Z}}q^{\frac{s_1^2}{2}}z_r^{-s_1}\cdot\frac{\theta(z_1)\theta(z_2)\cdots\theta(z_{r-1})\theta(z_1x_{2r-1})\theta(z_2x_{2r-3})\cdots\theta(z_{r-1}x_3)}{(\frac{x_{2r+1}}{z_1x_{2r-1}},\frac{z_1x_{2r-1}}{z_2x_{2r-3}},\dots,\frac{z_{r-2}x_{5}}{z_{r-1}x_{3}};q)_{\infty}}\right]\nonumber\\
   &=\frac{(-q^{\frac{1}{2}}x_{2r+1};q)_{\infty}}{(q;q)_{\infty}^{r}}\CT_{z_1,\dots,z_r}\left[\frac{\theta(z_1)\theta(z_2)\cdots\theta(z_{r})\theta(z_1x_{2r-1})\theta(z_2x_{2r-3})\cdots\theta(z_{r}x_1)}{(\frac{x_{2r+1}}{z_1x_{2r-1}},\frac{z_1x_{2r-1}}{z_2x_{2r-3}},\dots,\frac{z_{r-1}x_{3}}{z_{r}x_{1}};q)_{\infty}}\right] \nonumber \\
    &=\frac{(-q^{\frac{1}{2}}x_{2r+1};q)_{\infty}}{(q;q)_{\infty}^{r}}\CT_{z_1,\dots,z_r}\left[\frac{\theta(z_1)\theta(z_2)\cdots\theta(z_{r})\theta(z_{1}x_1)\theta(z_2x_{3}) \cdots \theta(z_{r}x_{2r-1})}{(\frac{x_{2r+1}}{z_rx_{2r-1}},\frac{z_rx_{2r-1}}{z_{r-1}x_{2r-3}},\dots,\frac{z_{2}x_{3}}{z_{1}x_{1}};q)_{\infty}}\right]. \label{proof-odd-consant}
\end{align}
Here for the last equality we interchange the symbol $z_j$ with $z_{r+1-j}$ for $j=1,2,\dots,r$, respectively.

Next, we are going to remove the constant term operators.
From \eqref{proof-even-constant} we have
\begin{align}
    &  \chi_{2r}(x_1,\dots,x_{2r};q) \nonumber \\
    &=\frac{(-q^{\frac{1}{2}}x_{2r};q)_{\infty}}{(q;q)_{\infty}^{r-1}}\CT_{z_1,\dots,z_{r}}\Bigg[\sum_{i\in \mathbb{Z}} q^{\frac{1}{2}i^2}z_1^{i} \sum_{n_1\geq 0} \frac{q^{\frac{1}{2}n_1^2}z_1^{-n_1}x_0^{-n_1}}{(q;q)_{n_1}} \sum_{n_2\geq 0} \frac{(\frac{z_2x_2}{z_1x_0})^{n_2}}{(q;q)_{n_2}} \nonumber \\
    &\quad \times \frac{\theta(z_2)\cdots \theta(z_r)\theta(z_2x_{2})\cdots \theta(z_{r}x_{2r-2})}{(\frac{x_{2r}}{z_rx_{2r-2}},\dots,\frac{z_{3}x_4}{z_{2}x_2};q)_\infty}\Bigg]\nonumber\\
    &=\frac{(-q^{\frac{1}{2}}x_{2r};q)_{\infty}}{(q;q)_{\infty}^{r-1}}\CT_{z_2,\dots,z_{r}}\Bigg[\sum_{n_1,n_2\geq0}\frac{q^{\frac{1}{2}n_1^2+\frac{1}{2}(n_1+n_2)^2}x_0^{-n_1-n_2}x_2^{n_2}}{(q;q)_{n_1}(q;q)_{n_2}}z_2^{n_2}  \nonumber  \\
    &\quad \quad \times   \frac{\theta(z_2)\cdots \theta(z_r)\theta(z_2x_2)\cdots \theta(z_r x_{2r-2})}{(\frac{z_3x_{4}}{z_2x_{2}},\dots,\frac{x_{2r}}{z_{r}x_{2r-2}};q)_{\infty}}\Bigg]\nonumber\\
     &=\frac{(-q^{\frac{1}{2}}x_{2r};q)_{\infty}}{(q;q)_{\infty}^{r-1}}\CT_{z_2,\dots,z_{r}}\Bigg[\sum_{n_1,n_2\geq0}\frac{q^{\frac{1}{2}n_1^2+\frac{1}{2}(n_1+n_2)^2}x_0^{-n_1-n_2}x_2^{n_2}}{(q;q)_{n_1}(q;q)_{n_2}}z_2^{n_2}   \times \sum_{i\in \mathbb{Z}} q^{\frac{1}{2}i^2}z_2^{-i}x_2^{-i} \nonumber  \\
    &\quad \quad \times \sum_{j\in \mathbb{Z}}   q^{\frac{1}{2}j^2}z_2^{-j} \sum_{n_3\geq 0} \frac{(\frac{z_3x_4}{z_2x_2})^{n_3}}{(q;q)_{n_3}} \times \frac{\theta(z_3)\cdots \theta(z_r)\theta(z_3x_4)\cdots \theta(z_r x_{2r-2})}{(\frac{z_4x_{6}}{z_3x_{4}},\dots,\frac{x_{2r}}{z_{r}x_{2r-2}};q)_{\infty}}\Bigg]\nonumber\\
     &=\frac{(-q^{\frac{1}{2}}x_{2r};q)_{\infty}}{(q;q)_{\infty}^{r-1}}\CT_{z_2,\dots,z_{r}}\Bigg[\sum_{n_1,n_2,n_3\geq0}\frac{q^{\frac{1}{2}n_1^2+\frac{1}{2}(n_1+n_2)^2}x_0^{-n_1-n_2}x_2^{n_2}(\frac{z_3x_4}{x_2})^{n_3}}{(q;q)_{n_1}(q;q)_{n_2}(q;q)_{n_3}}    \nonumber  \\
    &\quad \quad \times \sum_{\substack{i,j\in \mathbb{Z} \\ i+j=n_2-n_3}}   q^{\frac{1}{2}(i^2+j^2)} x_2^{-i} \times \frac{\theta(z_3)\cdots \theta(z_r)\theta(z_3x_4)\cdots \theta(z_r x_{2r-2})}{(\frac{z_4x_{6}}{z_3x_{4}},\dots,\frac{x_{2r}}{z_{r}x_{2r-2}};q)_{\infty}}\Bigg] \nonumber \\
     &=\frac{(-q^{\frac{1}{2}}x_{2r};q)_{\infty}}{(q;q)_{\infty}^{r-1}}\CT_{z_3,\dots,z_{r}}\Bigg[\sum_{n_1,n_2,n_3\geq0}\frac{q^{\frac{1}{2}n_1^2+\frac{1}{2}(n_1+n_2)^2+\frac{1}{4}(n_2-n_3)^2}x_0^{-n_1-n_2}x_2^{n_2-n_3}x_4^{n_3}}{(q;q)_{n_1}(q;q)_{n_2}(q;q)_{n_3}} \nonumber\\
    & \quad \times \Big(\sum_{\ell_2\in\mathbb{Z}}q^{(\ell_2-\frac{n_2-n_3}{2})^2}x_2^{-\ell_2}\Big)z_3^{n_3} \frac{\theta(z_3)\cdots \theta(z_r)\theta(z_3x_4)\cdots \theta(z_rx_{2r-2})}{(\frac{z_4x_6}{z_3x_4},\dots,\frac{x_{2r}}{z_rx_{2r-2}};q)_{\infty}}\Bigg].
    \label{proof-even-mid}
\end{align}
Here for the last equality we used the fact that when $i+j=n_{2}-n_3$ we have
\begin{align}\label{ij-id}
    \frac{1}{2}(i^2+j^2)=\left(\frac{i+j}{2}\right)^2+\left(\frac{i-j}{2}\right)^2=\frac{1}{4}(n_{2}-n_{3})^2+\left(i-\frac{n_{2}-n_3}{2}\right)^2.
\end{align}
Arguing similarly as above to eliminate the variables $z_3,z_4,\dots,z_r$, we eventually arrive at  \eqref{eq-reduce-even}.

For the odd rank case, starting from \eqref{proof-odd-consant} and arguing similarly as the even rank case to eliminate the variables $z_1,z_2,\dots,z_r$, we eventually arrive at \eqref{eq-reduce-odd}.
\end{proof}

Before we close this section, we point out that we may get some other representations for the generalized tadpole Nahm sums if we use the constant term method in a different way. Here we take the principal character $\chi_r(1,1,\dots,1;q)$ as an example to illustrate this phenomenon.
\begin{theorem}
\label{s2}
For $r>1$, we have
    \begin{align}\label{eq-thm-new-repn}
        \chi_r(1,1,\dots,1;q)=(q;q)_{\infty}^r\sum_{n_1,\dots,n_r\geq0}\frac{q^{X_r}(-q^{\frac{1}{2}};q)_{n_r}}{(q;q)_{n_1}^2\cdots (q;q)_{n_r}^2}.
    \end{align}
\end{theorem}
\begin{proof}
We define
\begin{align}
   & P_r=P_r(z_1,z_2,\dots,z_r;q) \nonumber \\
   &:=(-q^{\frac{1}{2}}z_1,\dots,-q^{\frac{1}{2}}z_r;q)_{\infty}\frac{(-q^{\frac{1}{2}}/{z_1},-q^{\frac{1}{2}}/{z_2},\dots,-q^{1/2}/z_r;q)_{\infty}}{(1/{z_1},z_1/z_2,\dots,z_{r-1}/z_r;q)_{\infty}}
\end{align}
where $1<|z_1|<|z_2|<\cdots<|z_{r-1}|<|z_r|$.
In view of \eqref{X-square} we have
    \begin{align}
      &  \chi_r(1,\dots,1;q)=\sum_{n_1,\dots,n_{r}\geq 0}\frac{q^{X_r}}{(q;q)_{n_1}\cdots (q;q)_{n_{r}}}\nonumber\\
        &=\CT_{z_1,\dots,z_{r}}\Big[\sum_{n_1\geq0}\frac{z_1^{-n_1}z_2^{n_1}}{(q;q)_{n_1}}\sum_{n_2\geq0}\frac{z_2^{-n_2}z_3^{n_2}}{(q;q)_{n_2}}\cdots \sum_{n_r\geq0}\frac{z_r^{-n_r}}{(q;q)_{n_r}} \nonumber \\
        &\qquad \times \sum_{\ell_1\in\mathbb{Z}}q^{\frac{\ell_1^2}{2}}z_1^{\ell_1}\sum_{\ell_2\in\mathbb{Z}}q^{\frac{\ell_2^2}{2}}z_2^{\ell_2}\cdots \sum_{\ell_r\in\mathbb{Z}}q^{\frac{\ell_r^2}{2}}z_r^{\ell_r}\Big]\nonumber\\
        &=\CT_{z_1,\dots,z_{r}}\left[\frac{\theta(z_1)\theta(z_2)\cdots \theta(z_r)}{({z_2}/{z_1},{z_3}/{z_2},\dots,{z_{r}}/{z_{r-1}},{1}/{z_r};q)_{\infty}}\right] \nonumber \\
        &=\CT_{z_1,\dots,z_{r}}\left[\frac{\theta(z_1)\theta(z_2)\cdots \theta(z_r)}{({z_0}/{z_1},{z_1}/{z_2},\dots,{z_{r-1}}/{z_r};q)_{\infty}}\right].
        \end{align}
Here for the last equality we interchange $z_j$ with $z_{r+1-j}$ for $j=1,2,\dots,r$ and we set $z_0=1$.

Recall the product expression of $\theta(z)$ in \eqref{JTP}. We have
\begin{align}
&  \chi_r(1,\dots,1;q) \nonumber \\
        &=(q;q)_{\infty}^r\CT_{z_1,\dots,z_{r}}\left[(-q^{\frac{1}{2}}z_1;q)_{\infty}\cdots (-q^{\frac{1}{2}}z_r;q)_{\infty}\frac{({-q^{\frac{1}{2}}}/{z_1};q)_{\infty}}{({z_0}/{z_1};q)_{\infty}}\cdots \frac{({-q^{\frac{1}{2}}}/{z_r};q)_{\infty}}{({z_{r-1}}/{z_r};q)_{\infty}}\right]\nonumber\\
        &=(q;q)_{\infty}^r\CT_{z_1,\dots,z_{r}}\left[P_{r-1}\sum_{i\geq0}\frac{q^{\frac{1}{2}i^2}z_r^i}{(q;q)_i}\sum_{j\geq0}\frac{({-q^{\frac{1}{2}}}/{z_{r-1}};q)_j}{(q;q)_j}\left(\frac{z_{r-1}}{z_r}\right)^j\right]\nonumber\\
        &=(q;q)_{\infty}^r\CT_{z_1,\dots,z_{r-1}}\left[P_{r-1}\sum_{n_1\geq0}\frac{q^{\frac{1}{2}n_{1}^2}(-\frac{q^{\frac{1}{2}}}{z_{r-1}};q)_{n_1}}{(q;q)_{n_1}^2}z_{r-1}^{n_1}\right]. \label{new-mid}
\end{align}
This eliminates the variable $z_r$.

Using \eqref{aq-finite} with $a=-q^{\frac{1}{2}}z_{r-1}$ and $n=n_1$ and using \eqref{euler-2} and \eqref{new-mid}, we deduce that
\begin{align}
&  \chi_r(1,\dots,1;q) \nonumber \\
 &=(q;q)_{\infty}^r\CT_{z_1,\dots,z_{r-1}}\left[P_{r-2}\sum_{j\geq0}\frac{(\frac{-q^{\frac{1}{2}}}{z_{r-2}};q)_j}{(q;q)_j}\left(\frac{z_{r-2}}{z_{r-1}}\right)^j\sum_{n_1\geq0}\frac{q^{n_1^2}(-q^{\frac{1}{2}-n_1}z_{r-1};q)_{\infty}}{(q;q)_{n_1}^2}\right] \nonumber \\
 &=(q;q)_{\infty}^r\CT_{z_1,\dots,z_{r-1}}\left[P_{r-2}\sum_{j\geq0}\frac{(\frac{-q^{\frac{1}{2}}}{z_{r-2}};q)_j}{(q;q)_j}\left(\frac{z_{r-2}}{z_{r-1}}\right)^j\sum_{n_1\geq0}\frac{q^{n_1^2}}{(q;q)_{n_1}^2} \sum_{n_2\geq 0} \frac{q^{\frac{1}{2}n_2^2-n_1n_2}z_{r-1}^{n_2}}{(q;q)_{n_2}}\right] \nonumber \\
        &=(q;q)_{\infty}^r\CT_{z_1,\dots,z_{r-2}}\left[P_{r-2}\sum_{n_1,n_2\geq0}\frac{q^{X_2}(\frac{-q^{\frac{1}{2}}}{z_{r-2}};q)_{n_2}z_{r-2}^{n_2}}{(q;q)_{n_1}^2(q;q)_{n_2}^2}\right].
\end{align}
This eliminates the variable $z_{r-1}$.
Repeating the above process to eliminate the variables $z_{r-2},z_{r-3},\dots, z_1$, respectively, we eventually arrive at
\begin{align}
&  \chi_r(1,\dots,1;q)=(q;q)_{\infty}^r \sum_{n_1,\dots,n_{r}\geq0}\frac{q^{X_{r}}(\frac{-q^{\frac{1}{2}}}{z_0};q)_{n_{r}}z_0^{n_{r}}}{(q;q)_{n_1}^2...(q;q)_{n_{r}}^2}.
\end{align}
Since $z_0=1$, we obtain the desired identity.
\end{proof}
It is not clear to us whether \eqref{eq-thm-new-repn} is helpful or not to prove Conjecture \ref{conj-CMP}.

\section{The rank four tadpole Nahm sums}\label{sec-rank4}
\subsection{Proof of Theorem \ref{thm-dual-Zagier}}\label{sec-dual}
We first evaluate some alternating triple sums which will play a key role in the proof of Theorem \ref{thm-dual-Zagier}.
\begin{lemma}\label{lem-G}
For $k=1,2,3,4$ we define
\begin{align}
    &G_k(q):=\sum_{m,n,r\geq0}\frac{(-1)^mq^{(n-m)^2+(n-r)^2+(m,n,r)\beta_k}}{(q^4;q^4)_m(q^4;q^4)_n(q^4;q^4)_r}, \label{Gk-defn}\\
    &\widetilde{G}_k(q):=\sum_{m,n,r\geq0}\frac{(-1)^nq^{(n-m)^2+(n-r)^2+(m,n,r)\beta_k}}{(q^4;q^4)_m(q^4;q^4)_n(q^4;q^4)_r} \label{wGk-defn}
\end{align}
where $\beta_1=(2,0,0)^\mathrm{T}$, $\beta_2=(2,-2,2)^\mathrm{T}$, $\beta_3=(2,0,2)^\mathrm{T}$ and $\beta_4=(0,0,2)^\mathrm{T}$. We have
\begin{align}
    G_1(q)&=\frac{J_4^2}{J_2J_8}+2q\frac{J_2^2J_8^3J_{12}^2}{J_4^6J_6}, \label{eq-G1} \\
    \widetilde{G}_1(q)&=\frac{J_2^3J_{8}^2J_{12}^2}{J_4^6J_{24}}, \label{eq-wG1} \\
    G_2(q)&= 2\frac{J_2J_6^2J_8^3}{J_4^5J_{12}}, \label{eq-G2}\\
    \widetilde{G}_2(q)&=-2q\frac{J_2^3J_8J_{24}^2}{J_4^5J_{12}},  \label{eq-wG2} \\
    G_3(q)&=\frac{J_4^3J_{12}^2}{J_2J_6J_8^3},   \label{eq-G3}\\
    \widetilde{G}_3(q)&=\frac{J_2J_6J_{24}}{J_8^2J_{12}},    \label{eq-wG3} \\
    G_4(q)&=\frac{J_4^2}{J_2J_8}-2q\frac{J_2^2J_8^3J_{12}^2}{J_4^6J_6},  \label{eq-G4} \\
    \widetilde{G}_4(q)&=\frac{J_2^3J_{8}^2J_{12}^2}{J_4^6J_{24}}.  \label{eq-wG4}
\end{align}
\end{lemma}
\begin{proof}
For convenience, we denote $\beta_k=(a_k,b_k,c_k)^\mathrm{T}$ ($k=1,2,3,4$). Summing over $n$ first using \eqref{euler-2}, we obtain
\begin{align}
        &G_k(q)=\sum_{m,r\geq0}\frac{(-1)^mq^{m^2+r^2+a_km+c_kr}(-q^{2+b_k-2m-2r};q^4)_{\infty}}{(q^4;q^4)_m(q^4;q^4)_r} \nonumber \\
        &=\sum_{n\geq 0} q^{n^2}(-q^{2+b_k-2n};q^4)_\infty \sum_{\substack{m,r\geq 0 \\ m+r=n}}\frac{(-1)^mq^{-2mr+a_km+c_kr}}{(q^4;q^4)_m(q^4;q^4)_r} \nonumber \\
        &=G_k^{(0)}(q)+G_k^{(1)}(q), \\
        &\widetilde{G}_k(q)=\sum_{m,r\geq0}\frac{q^{m^2+r^2+a_km+c_kr}(q^{2+b_k-2m-2r};q^4)_{\infty}}{(q^4;q^4)_m(q^4;q^4)_r} \nonumber \\
        &=\sum_{n\geq 0} q^{n^2}(q^{2+b_k-2n};q^4)_\infty \sum_{\substack{m,r\geq 0 \\ m+r=n}}\frac{q^{-2mr+a_km+c_kr}}{(q^4;q^4)_m(q^4;q^4)_r}\\
        &=\widetilde{G}_k^{(0)}(q)+\widetilde{G}_k^{(1)}(q)
    \end{align}
where for $i=0,1$ we have
\begin{align}
    &G_k^{(i)}(q)=\sum_{\substack{ n\geq 0 \\ n\equiv i \!\!\pmod{2}}} q^{n^2}(-q^{2+b_k-2n};q^4)_\infty \sum_{\substack{m,r\geq 0 \\ m+r=n}}\frac{(-1)^mq^{-2mr+a_km+c_kr}}{(q^4;q^4)_m(q^4;q^4)_r}, \label{Gki-mid} \\
    &\widetilde{G}_k^{(i)}(q)=\sum_{\substack{n\geq 0 \\ n \equiv i \!\! \pmod{2}}} q^{n^2}(q^{2+b_k-2n};q^4)_\infty \sum_{\substack{m,r\geq 0 \\ m+r=n}}\frac{q^{-2mr+a_km+c_kr}}{(q^4;q^4)_m(q^4;q^4)_r}. \label{wGki-mid}
\end{align}

For $k=1,2,3,4$ and $i=0,1$, we evaluate $G_k^{(i)}(q)$ and $\widetilde{G}_k^{(i)}(q)$ one by one.

(1) When $k=1$, using \eqref{finite-sum-1}--\eqref{finite-sum-4}, \eqref{finite-6}, \eqref{finite-3}, \eqref{finite-add} and \eqref{finite-1} we deduce that
\begin{align}
    G_1^{(0)}(q)&=(-q^2;q^4)_{\infty}, \label{G10-exp}\\
    G_1^{(1)}(q)&=2q(-q^4;q^4)_{\infty}\sum_{i\geq0}\frac{(-1)^iq^{2i}(q^2;q^4)_i(q^2;q^4)_{i+1}(-q^4;q^4)_i}{(q^4;q^4)_{2i+1}}, \label{G11-exp}\\
    \widetilde{G}_1^{(0)}(q)&=(q^2;q^4)_{\infty}\sum_{i\geq0}\frac{(-1)^iq^{2i}(-1;q^4)_i(q^2;q^4)_i(-q^4;q^4)_i}{(q^4;q^4)_{2i}}, \label{wG10-exp}\\
    \widetilde{G}_1^{(1)}(q)&=0. \label{wG11-exp}
\end{align}

Taking $(a,b,c,t,q)$ as $(q,0,-q^3,-q,q^2)$ in \eqref{Heine}, we deduce that
\begin{align}
    &\sum_{i\geq0}\frac{(-q)^i(q,-q^2;q^2)_i(q;q^2)_{i+1}}{(q^2;q^2)_{2i+1}}=\frac{1}{1+q}\sum_{i\geq0}\frac{(q;q^2)_i(-q)^i}{(-q^3;q^2)_i(q^2;q^2)_i}\nonumber\\
    &=\frac{(-q^2;q^2)_{\infty}}{(-q;q^2)_{\infty}^2}\sum_{i\geq0}\frac{(-q;q^2)_iq^{i^2+2i}}{(q^4;q^4)_i}=\frac{J_1^2J_4^2J_{6}^2}{J_2^5J_3}. \quad \text{(by \eqref{Ramanujan[10 Entry4.2.11]})}\label{G11}
\end{align}
Substituting \eqref{G11} into \eqref{G11-exp}  and combining with \eqref{G10-exp}, we obtain \eqref{eq-G1}.

Similarly, taking $(a,b,c,t,q)$ as $(-1,0,-q,-q,q^2)$  in \eqref{Heine},  we deduce that
\begin{align}
    &\sum_{i\geq0}\frac{(-q)^i(-1,q,-q^2;q^2)_i}{(q^2;q^2)_{2i}}=\sum_{i\geq0}\frac{(-1;q^2)_i(-q)^i}{(-q;q^2)_i(q^2;q^2)_i} \nonumber \\
    &=\frac{(q;q^2)_{\infty}}{(-q;q^2)_{\infty}^2}\sum_{i\geq0}\frac{(-q;q^2)_iq^{i^2}}{(q,q^2;q^2)_i}=\frac{J_1^2J_4^2J_6^2}{J_2^5J_{12}}. \quad \text{(by \eqref{S.29})} \label{wG10}
\end{align}
Substituting \eqref{wG10} into \eqref{wG10-exp} and combining with \eqref{wG11-exp}, we obtain \eqref{eq-wG1}.

(2) When $k=2$, using \eqref{finite-sum-5}--\eqref{finite-sum-7}, \eqref{finite-3}, \eqref{finite-4}, \eqref{finite-1} and \eqref{finite-2} we deduce that
    \begin{align}
        G_2^{(0)}(q)&=2(-q^4;q^4)_{\infty}\sum_{i\geq0}\frac{(-1)^iq^{2i}(-q^4;q^4)_i(q^2;q^4)_i^2}{(q^4;q^4)_{2i}}, \label{G20-exp}\\
        G_2^{(1)}(q)&=0, \label{G21-exp}\\
        \widetilde{G}_2^{(0)}(q)&=0, \label{wG20-exp}\\
        \widetilde{G}_2^{(1)}(q)&=-2q(q^2;q^4)_{\infty}\sum_{i\geq0}\frac{(-1)^iq^{2i}(q^2;q^4)_{i+1}(-q^4;q^4)_i^2}{(q^4;q^4)_{2i+1}}. \label{wG21-exp}
    \end{align}
Setting $(a,b,c,t,q)$ as $(q,0,-q,-q,q^2)$ in \eqref{Heine}, we deduce that
\begin{align}
    &\sum_{i\geq0}\frac{(-q)^i(-q^2;q^2)_i(q;q^2)_i^2}{(q^2;q^2)_{2i}}=\sum_{i\geq0}\frac{(-q)^i(q;q^2)_i}{(-q;q^2)_i(q^2;q^2)_i}\nonumber \\
    &=\frac{(-q^2;q^2)_{\infty}}{(-q;q^2)_{\infty}^2}\sum_{i\geq0}\frac{(-q;q^2)_iq^{i^2}}{(q^4;q^4)_i}=\frac{J_1J_3^2J_4^2}{J_2^4J_6}. \quad \text{(by \eqref{S.25})} \label{G20}
\end{align}
Substituting \eqref{G20} into \eqref{G20-exp} and combining with \eqref{G21-exp}, we obtain \eqref{eq-G2}.

Similarly, setting $(a,b,c,t,q)$ as $(-q^2,0,q^3,q,q^2)$ in \eqref{Heine}, we deduce that
\begin{align}
    &\sum_{i\geq0}\frac{(-q)^i(q;q^2)_{i+1}(-q^2;q^2)_i^2}{(q^2;q^2)_{2i+1}}=\frac{1}{1+q}\sum_{i\geq0}\frac{(-q)^i(-q^2;q^2)_i}{(-q^3;q^2)_i(q^2;q^2)_i}\nonumber\\&=\frac{(q;q^2)_{\infty}}{(-q;q^2)_{\infty}^2}\sum_{i\geq0}\frac{q^{i^2+2i}(-q;q^2)_i}{(q;q^2)_{i+1}(q^2;q^2)_i}=\frac{J_1^2J_4J_{12}^2}{J_2^4J_6}. \quad \text{(by \eqref{S.50})}\label{wG21}
\end{align}
Substituting \eqref{wG21} into \eqref{wG21-exp} and combining with \eqref{wG20-exp}, we obtain \eqref{eq-wG2}.

(3) When $k=3$, using \eqref{finite-sum-5}--\eqref{finite-sum-7}, \eqref{finite-6}, \eqref{finite-3} and \eqref{finite-add} we deduce that
\begin{align}
    G_3^{(0)}(q)&=(-q^2;q^4)_{\infty}\sum_{i\geq0}\frac{(-1)^iq^{4i}(-q^2;q^4)_i(q^2;q^4)_i^2}{(q^4;q^4)_{2i}}, \label{G30-exp}\\
    G_3^{(1)}(q)&=0, \label{G31-exp}\\
    \widetilde{G}_3^{(0)}(q)&=(q^2;q^4)_{\infty}\sum_{i\geq0}\frac{(-1)^iq^{4i}(q^2;q^4)_i(-q^2;q^4)_i^2}{(q^4;q^4)_{2i}}, \label{wG30-exp}\\
    \widetilde{G}_3^{(1)}(q)&=0. \label{wG31-exp}
\end{align}
Setting $(a,b,c,t,q)$ as $(-q,0,-q^2,-q^2,q^2)$ in \eqref{Heine} and using \eqref{S.28}, we deduce that
\begin{align}
    &\sum_{i\geq 0} \frac{(-1)^iq^{2i}(q;q^2)_i(-q;q^2)_i^2}{(q^2;q^2)_{2i}}=\sum_{i\geq0}\frac{(-1)^iq^{2i}(-q;q^2)_i}{(q^4;q^4)_i} \nonumber \\
    &=\frac{(q;q^2)_{\infty}}{(-q^2;q^2)_{\infty}^2}\sum_{i\geq0}\frac{(-q^2;q^2)_iq^{i(i+1)}}{(q;q^2)_{i+1}(q^2;q^2)_i}=\frac{J_2J_3J_{12}}{J_4^2J_6}.\label{G3}
\end{align}
Substituting \eqref{G3} with $q$ replaced by $-q^2$ into \eqref{G30-exp} and combining with \eqref{G31-exp}, we obtain \eqref{eq-G3}.

Similarly, substituting \eqref{G3}  into \eqref{wG30-exp} and combining with \eqref{wG31-exp}, we obtain \eqref{eq-wG3}.

(4) Finally, we treat the case $k=4$. Note that by interchanging $a_1$ with $b_1$ in $\beta_1$ we get $\beta_4$. Using \eqref{Gki-mid} and \eqref{wGki-mid} it is easy to see that
\begin{align}
G_4^{(0)}(q)&=G_1^{(0)}(q)=(-q^2;q^4)_\infty, \label{G40-exp} \\
G_4^{(1)}(q)&=-G_1^{(1)}(q)=-2q\frac{J_2^2J_8^3J_{12}^2}{J_4^6J_6}. \label{G41-exp}
\end{align}
This proves \eqref{eq-G4}. Interchanging $m$ with $r$ in the definition of $\widetilde{G}_1(q)$ we obtain $\widetilde{G}_4(q)$. Hence $\widetilde{G}_4(q)=\widetilde{G}_1(q)$ and  we obtain \eqref{eq-wG4} from \eqref{eq-wG1}.
\end{proof}

\begin{proof}[Proof of Theorem \ref{thm-dual-Zagier}]
We define
\begin{align}\label{Zk-defn}
Z(\alpha;q):=\sum_{n_1,n_2,n_3\geq0}\frac{q^{2n_1^2+2(n_1+n_2)^2+(n_2-n_3)^2+4{n}^\mathrm{T} \alpha}}{(q^4;q^4)_{n_1}(q^4;q^4)_{n_2}(q^4;q^4)_{n_3}}.
\end{align}

For convenience, we replace the indices $n_1,n_2,n_3$ by $m,n,r$, respectively. If $\alpha=(a,b,c)^\mathrm{T}$ with $a-b-c<\frac{1}{2}$, then we have
\begin{align}
   & Z(\alpha;q^{\frac{1}{4}})=\sum_{m,n,r\geq0}\frac{q^{\frac{1}{2}m^2+\frac{1}{2}(m+n)^2+\frac{1}{4}(n-r)^2+am+bn+cr}}{(q;q)_{m}(q;q)_{n}(q;q)_{r}}\nonumber\\
   & =\CT_{z}\Big[\sum_{m\geq 0}\frac{q^{\frac{1}{2}m^2+am}z^{m}}{(q;q)_{m}}
    \sum_{n,r\geq 0}\frac{q^{\frac{1}{4}(n-r)^2+bn+cr}z^{n}}{(q;q)_{n}(q;q)_{r}}\sum_{i\in\mathbb{Z}}q^{\frac{1}{2}i^2}z^{-i}\Big]\nonumber\\
&=\CT_{z}\Big[(-q^{\frac{1}{2}+a}z;q)_\infty \sum_{n,r\geq 0} \frac{q^{\frac{1}{4}(n-r)^2+bn+cr}z^n}{(q;q)_n(q;q)_r}  \sum_{i\in \mathbb{Z}}q^{\frac{1}{2}i^2}z^{-i}\Big] \nonumber \\
&=\frac{1}{(q;q)_\infty}\CT_{z}\Big[\frac{(-q^{\frac{1}{2}+a}z,-q^{\frac{1}{2}-a}/z,q;q)_\infty}{(-q^{\frac{1}{2}-a}/z;q)_\infty}\sum_{n,r\geq 0} \frac{q^{\frac{1}{4}(n-r)^2+bn+cr}z^n}{(q;q)_n(q;q)_r}  \sum_{i\in \mathbb{Z}}q^{\frac{1}{2}i^2}z^{-i}\Big] \nonumber \\
   & =\frac{1}{(q;q)_{\infty}}\CT_{z}\Big[\Big(\sum_{i\in\mathbb{Z}}q^{\frac{1}{2}i^2}z^{-i}\Big)\Big(\sum_{j\in \mathbb{Z}} q^{\frac{1}{2}j^2-aj}z^{-j}\Big)\sum_{n,r\geq0}\frac{q^{\frac{1}{4}(n-r)^2+bn+cr}z^{n}}{(q;q)_{n}(q;q)_{r}} \nonumber \\
   &\qquad \qquad \qquad \times \sum_{m\geq 0}\frac{(-1)^{m}q^{(\frac{1}{2}-a)m}z^{-m}}{(q;q)_{m}}\Big] \qquad (|q|^{\frac{1}{2}-a}<|z|<|q|^{-b-c})\nonumber\\
  &  =\frac{1}{(q;q)_{\infty}}\sum_{\substack{m,n,r\geq0\\i,j\in\mathbb{Z}\\n-m=i+j}}\frac{(-1)^{m}q^{\frac{1}{2}(i^2+j^2)-aj+(\frac{1}{2}-a)m+\frac{1}{4}(n-r)^2+bn+cr}}{(q;q)_{m}(q;q)_{n}(q;q)_{r}}\nonumber\\
   & =\frac{q^{-\frac{1}{4}a^2}}{(q;q)_{\infty}}\sum_{m,n,r\geq0}\frac{(-1)^{m}q^{\frac{1}{4}(m-n)^2+\frac{1}{4}(n-r)^2+\frac{1}{2}(1-a)m+(b-\frac{1}{2}a)n+cr}}{(q;q)_{m}(q;q)_{n}(q;q)_{r}}\sum_{i\in\mathbb{Z}}q^{(i-\frac{n-m}{2}+\frac{a}{2})^2}. \label{Z-CT}
\end{align}
Here for the last equality we used \eqref{ij-id}.
Note that each of the the series involving $z$ in the above equation is absolutely convergent when $|q|^{\frac{1}{2}-a}<|z|<|q|^{-b-c}$ according to Lemma \ref{lem-convergence-double}. Hence the constant term method is valid when $z$ is in this annulus.

For convenience, we denote $Z_k(q)=Z(\alpha_k;q)$ where the vectors
\begin{equation}\label{alpha1234}
\begin{split}
    &\alpha_1=(0,0,0)^\mathrm{T}, \quad \alpha_2=(0,-1/{2},{1}/{2})^\mathrm{T}, \\
    &\alpha_3=(0,0,{1}/{2})^\mathrm{T}, \quad
\alpha_4=(1,{1}/{2},{1}/{2})^\mathrm{T}.
\end{split}
\end{equation}
From \eqref{Z-CT} we have for $k=1,2,3$ that
\begin{align}\label{Z123}
        Z_k(q)=\frac{1}{(q^4;q^4)_{\infty}}\sum_{m,n,r\geq0}\frac{(-1)^mq^{(n-m)^2+(n-r)^2+(m,n,r)\beta_k}}{(q^4;q^4)_m(q^4;q^4)_n(q^4;q^4)_r}\sum_{i\in\mathbb{Z}}q^{4(i-\frac{n-m}{2})^2}
\end{align}
and
\begin{align}\label{Z456}
    &Z_4(q)=\frac{q^{-1}}{(q^4;q^4)_{\infty}}\sum_{m,n,r\geq0}\frac{(-1)^mq^{(n-m)^2+(n-r)^2+ (m,n,r)\beta_4}}{(q^4;q^4)_m(q^4;q^4)_n(q^4;q^4)_r}\sum_{i\in\mathbb{Z}}q^{4(i+\frac{1}{2}-\frac{n-m}{2})^2}.
\end{align}
Here the vectors $\beta_k$ ($k=1,2,3,4$) are given in Lemma \ref{lem-G}.

We define
\begin{align}
        &L_k^{(0)}(q):=\sum_{\substack{m,n,r\geq0\\n-m\ \text{even}}}\frac{(-1)^mq^{(n-m)^2+(n-r)^2+(m,n,r)\beta_k}}{(q^4;q^4)_m(q^4;q^4)_n(q^4;q^4)_r},\\
        &L_k^{(1)}(q):=\sum_{\substack{m,n,r\geq0\\n-m\ \text{odd}}}\frac{(-1)^mq^{(n-m)^2+(n-r)^2+(m,n,r)\beta_k}}{(q^4;q^4)_m(q^4;q^4)_n(q^4;q^4)_r}.
    \end{align}
Recall the functions $G_k(q)$ and $\widetilde{G}_k(q)$ defined in \eqref{Gk-defn} and \eqref{wGk-defn}. Note that for $1\leq k\leq 4$ we have
\begin{align}
L_k^{(0)}(q)=\frac{1}{2}\left(G_k(q)+\widetilde{G}_k(q)\right), \quad
L_k^{(1)}(q)=\frac{1}{2}\left(G_k(q)-\widetilde{G}_k(q)\right).
\end{align}

From \eqref{Z123}, \eqref{theta0-defn} and \eqref{theta1-defn} we have for $k=1,2,3$ that
\begin{align}
    &Z_k(q)=\frac{1}{(q^4;q^4)_{\infty}}\Big(L_k^{(0)}(q)\sum_{i\in\mathbb{Z}}q^{(2i)^2}+L_k^{(1)}(q)\sum_{i\in\mathbb{Z}}q^{(2i+1)^2}\Big) \nonumber \\
    &=\frac{1}{2(q^4;q^4)_\infty}\Big(\big(G_k(q)+\widetilde{G}_k(q)\big)\theta_0+\big(G_k(q)-\widetilde{G}_k(q) \big)\theta_1  \Big). \label{Z123-split}
\end{align}
Similarly, from \eqref{Z456}, \eqref{theta0-defn} and \eqref{theta1-defn} we have
\begin{align}
    &Z_4(q)=\frac{q^{-1}}{(q^4;q^4)_{\infty}}\Big(L_4^{(0)}(q)\sum_{i\in\mathbb{Z}}q^{(2i+1)^2}+L_4^{(1)}(q)\sum_{i\in\mathbb{Z}}q^{(2i)^2}\Big) \nonumber \\
    &=\frac{q^{-1}}{2(q^4;q^4)_\infty}\Big(\big(G_4(q)+\widetilde{G}_4(q)\big)\theta_1+\big(G_4(q)-\widetilde{G}_4(q) \big)\theta_0\Big). \label{Z456-split}
\end{align}
Substituting \eqref{eq-G1}--\eqref{eq-wG4} into \eqref{Z123-split} and \eqref{Z456-split}, and then using the method in \cite{Frye-Garvan} to verify theta function identities, we obtain \eqref{dZ1}--\eqref{dZ4}.
\end{proof}

\begin{rem}
When we apply the above method to Conjecture \ref{conj-dZ5}, we face some convergence issues. To be specific, if we set
\begin{align}\label{alpha5}
\alpha_5:=(1,1/2,0)^\mathrm{T} \quad \text{and} \quad Z_5(q)=Z(\alpha_5;q),
\end{align}
then from \eqref{Z-CT} it appears that
\begin{align}
Z_5(q)\overset{?}{=}\frac{q^{-1}}{(q^4;q^4)_\infty}\sum_{m,n,r\geq0}\frac{(-1)^mq^{(m-n)^2+(n-r)^2}}{(q^4;q^4)_{m}(q^4;q^4)_n(q^4;q^4)_r}\sum_{i\in\mathbb{Z}}q^{(2i+1-(n-m))^2}.
\end{align}
However, the series defined by the triple sums over $m,n,r$ in the right side does not converge absolutely and hence is not well-defined.
\end{rem}


\subsection{Proof of Theorem \ref{thm-rank4}}\label{subsec-rank4}

Now we have all ingredients to give modular expressions for the rank four tadpole Nahm sums stated in Theorem \ref{thm-rank4}.

\begin{proof}[Proof of Theorem \ref{thm-rank4}]
We start with the following formula which is the rank four case of \eqref{eq-reduce-even}:
   \begin{align}
        \chi_4(x_1,x_2,1/x_2,x_4;q)&=\frac{(-q^{\frac{1}{2}}x_4;q)_{\infty}}{(q;q)_{\infty}}\sum_{m,n,r\geq0}\frac{q^{\frac{1}{2}m^2+\frac{1}{2}(m+n)^2+\frac{1}{4}(n-r)^2}x_1^{m}(x_1x_2)^{n}({x_4}/{x_2})^{r}}{(q;q)_m(q;q)_n(q;q)_r}\nonumber \\
        &\quad \times \sum_{i\in\mathbb{Z}}q^{(i-\frac{n-r}{2})^2}x_2^{-i}. \label{prop-rank4-start}
    \end{align}
Note that
\begin{align}\label{eq-rank4-square}
        \sum_{i\in\mathbb{Z}}q^{(i-\frac{n-r}{2})^2-bi}=\sum_{i\in\mathbb{Z}}q^{(i-\frac{b}{2}-\frac{n-r}{2})^2-\frac{1}{4}b^2-\frac{1}{2}b(n-r)}.
    \end{align}

Setting $(x_1,x_2,x_4)=(q^a,q^b,q^c)$ in \eqref{prop-rank4-start} and using \eqref{eq-rank4-square}, we deduce that
\begin{align}
        &\chi_4(q^a,q^b,q^{-b},q^c;q)=\frac{(-q^{\frac{1}{2}+c};q)_{\infty}}{(q;q)_{\infty}}\nonumber\\
        &\sum_{m,n,r\geq0}\frac{q^{\frac{1}{2}m^2+\frac{1}{2}(m+n)^2+\frac{1}{4}(n-r)^2+am+(a+\frac{b}{2})n+(c-\frac{b}{2})r-\frac{1}{4}b^2}}{(q;q)_m(q;q)_n(q;q)_r}\sum_{i\in\mathbb{Z}}q^{(i-\frac{b}{2}-\frac{n-r}{2})^2} .
    \end{align}
In particular, we have
\begin{align}
    \chi_4(1,1,1,1;q)&=\frac{(-q^{\frac{1}{2}};q)_{\infty}}{(q;q)_{\infty}}\sum_{m,n,r\geq0}\frac{q^{\frac{1}{2}m^2+\frac{1}{2}(m+n)^2+\frac{1}{4}(n-r)^2}}{(q;q)_m(q;q)_n(q;q)_r}\sum_{i\in\mathbb{Z}}q^{(i-\frac{n-r}{2})^2}, \label{chi4-1}\\
   \chi_4(1,q^{-1},q,1;q)&=q^{-\frac{1}{4}}\frac{(-q^{\frac{1}{2}};q)_{\infty}}{(q;q)_{\infty}}\sum_{m,n,r\geq0}\frac{q^{\frac{1}{2}m^2+\frac{1}{2}(m+n)^2+\frac{1}{4}(n-r)^2-\frac{1}{2}n+\frac{1}{2}r}}{(q;q)_m(q;q)_n(q;q)_r} \nonumber \\
   &\qquad \qquad \times \sum_{i\in\mathbb{Z}}q^{(i+\frac{1-n+r}{2})^2}, \label{chi4-2}\\
    \chi_4(1,1,1,q^{\frac{1}{2}};q)&=\frac{(-q;q)_{\infty}}{(q;q)_{\infty}}\sum_{m,n,r\geq0}\frac{q^{\frac{1}{2}m^2+\frac{1}{2}(m+n)^2+\frac{1}{4}(n-r)^2+\frac{1}{2}r}}{(q;q)_m(q;q)_n(q;q)_r}\sum_{i\in\mathbb{Z}}q^{(i-\frac{n-r}{2})^2}, \label{chi4-3} \\
    \chi_4(q,q^{-1},q,1;q)&=q^{-\frac{1}{4}}\frac{(-q^{\frac{1}{2}};q)_{\infty}}{(q;q)_{\infty}}\sum_{m,n,r\geq0}\frac{q^{\frac{1}{2}m^2+\frac{1}{2}(m+n)^2+\frac{1}{4}(n-r)^2+m+\frac{1}{2}n+\frac{1}{2}r}}{(q;q)_m(q;q)_n(q;q)_r} \nonumber \\
    &\qquad \qquad \times \sum_{i\in\mathbb{Z}}q^{(i+\frac{1-n+r}{2})^2},  \label{chi4-5}  \\
    \chi_4(q,q^{-1},q,q^{-\frac{1}{2}};q)&=q^{-\frac{1}{4}}\frac{(-1;q)_{\infty}}{(q;q)_{\infty}}\sum_{m,n,r\geq0}\frac{q^{\frac{1}{2}m^2+\frac{1}{2}(m+n)^2+\frac{1}{4}(n-r)^2+m+\frac{1}{2}n}}{(q;q)_m(q;q)_n(q;q)_r} \nonumber \\
    &\qquad \qquad \times \sum_{i\in\mathbb{Z}}q^{(i+\frac{1-n+r}{2})^2}. \label{chi4-4}
\end{align}
The six series above are related to the Nahm sums in Theorem \ref{thm-dual-Zagier}. We define
\begin{align}
    Z_k^{(0)}(q):=\sum_{\substack{m,n,r\geq0\\n-r\ \text{even}}}\frac{q^{2m^2+2(m+n)^2+(n-r)^2+4(m,n,r)\alpha_k}}{(q^4;q^4)_m(q^4;q^4)_n(q^4;q^4)_r}, \label{Zk0}\\
    Z_k^{(1)}(q):=\sum_{\substack{m,n,r\geq0\\n-r\ \text{odd}}}\frac{q^{2m^2+2(m+n)^2+(n-r)^2+4(m,n,r)\alpha_k}}{(q^4;q^4)_m(q^4;q^4)_n(q^4;q^4)_r}. \label{Zk1}
\end{align}
Here the vectors $\alpha_k$ ($1\leq k\leq 5$) are given in \eqref{alpha1234} and \eqref{alpha5}.
Note that for $1\leq k\leq 5$, $4(m,n,r)\alpha_k$ is always even. Hence the parity of the exponent of $q$ in the numerator of \eqref{Zk0} and \eqref{Zk1} is the same with $n-r$. It follows that
\begin{align}
Z_k^{(0)}(q)=\frac{1}{2}\big(Z_k(q)+Z_k(-q)\big),\quad Z_k^{(1)}(q)=\frac{1}{2}\big(Z_k(q)-Z_k(-q)\big) \label{Zk01-exp}
\end{align}
where $Z_k(q)=Z(\alpha_k;q)$ is defined in \eqref{Zk-defn}.
From \eqref{chi4-1}--\eqref{chi4-4} we deduce that
\begin{align}
    \chi_4(1,1,1,1;q^4)&=\frac{(-q^{2};q^4)_{\infty}}{(q^4;q^4)_{\infty}}\big(Z_1^{(0)}(q)\theta_0+Z_1^{(1)}(q)\theta_1\big), \label{chi4-1-final}\\
    \chi_4(1,q^{-4},q^4,1;q^4)&=\frac{q^{-1}(-q^{2};q^4)_{\infty}}{(q^4;q^4)_{\infty}}\big(Z_2^{(0)}(q)\theta_1+Z_2^{(1)}(q)\theta_0\big), \label{chi4-2-final}\\
    \chi_4(1,1,1,q^{2};q^4)&=\frac{(-q^4;q^4)_{\infty}}{(q^4;q^4)_{\infty}}\big(Z_3^{(0)}(q)\theta_0+Z_3^{(1)}(q)\theta_1\big), \label{chi4-3-final}\\
   \chi_4(q^4,q^{-4},q^4,1;q^4)&=\frac{q^{-1}(-q^{2};q^4)_{\infty}}{(q^4;q^4)_{\infty}}\big(Z_4^{(0)}(q)\theta_1+Z_4^{(1)}(q)\theta_0\big), \label{chi4-4-final} \\
   \chi_4(q^4,q^{-4},q^4,q^{-2};q^4)&=\frac{q^{-1}(-1;q^4)_{\infty}}{(q^4;q^4)_{\infty}}\big(Z_5^{(0)}(q)\theta_1+Z_5^{(1)}(q)\theta_0\big). \label{chi4-5-final}
\end{align}
Here the functions $\theta_0$ and $\theta_1$ are defined in \eqref{theta0-defn} and \eqref{theta1-defn}.
Substituting \eqref{dZ1}--\eqref{dZ5} into \eqref{Zk01-exp} and then substituting the results into \eqref{chi4-1-final}--\eqref{chi4-5-final} and using the method in \cite{Frye-Garvan} to verify theta function identities, we prove \eqref{eq-4-1}--\eqref{eq-4-5}.
\end{proof}

Using the identities in Theorem \ref{thm-rank4}, it is easy to check that $q^C\chi_4(q^a,q^b,q^{-b},q^c;q^2)$ is modular (e.g.\ using the method in \cite{Frye-Garvan}) for the choices of $(a,b,c,C)$ given in Table \ref{tab-rank4}.
\begin{table}[H]
\renewcommand{\arraystretch}{1.5}
    \centering
    \begin{tabular}{c|ccccc}
    \hline
        $(a,b,c)$ & $(0,0,0)$ & $(0,-2,0)$  & $(0,0,1)$  & $(2,-2,0)$  & $(2,-2,-1)$ \\
        \hline
        $C$ & $-1/4$  & $5/12$  & $0$ & $3/4$  & $2/3$\\
        \hline
    \end{tabular}
    \caption{Values of $(a,b,c,C)$ such that $q^C\chi_4(q^a,q^b,q^{-b},q^c;q^2)$ is modular}
    \label{tab-rank4}
\end{table}

\section{The rank five tadpole Nahm sums}\label{sec-rank5}
We start with the following consequence of Theorem \ref{thm-rank-reduction} which provides a new expression for some rank five tadpole Nahm sums:
    \begin{align}\label{eq-rank5-start}
        \chi_5(x_1,{1}/{x_1},x_3,{1}/{x_3},x_5;q)&=\frac{(-q^{\frac{1}{2}}x_5;q)_{\infty}}{(q;q)_{\infty}^2}\sum_{m,n\geq0}\frac{q^{\frac{1}{4}m^2+\frac{1}{4}(m-n)^2}(\frac{x_3}{x_1})^m(\frac{x_5}{x_3})^n}{(q;q)_m(q;q)_n} \nonumber \\
        & \quad \times\sum_{i,j\in\mathbb{Z}}q^{(i+\frac{m}{2})^2+(j-\frac{m-n}{2})^2}x_1^{-i}x_3^{-j}.
    \end{align}
Setting $(x_1,x_3,x_5)=(q^a,q^b,q^c)$ in \eqref{eq-rank5-start} and noting that
\begin{align}
q^{(i+\frac{m}{2})^2-ai}=q^{(i-\frac{a}{2}+\frac{m}{2})^2+\frac{am}{2}-\frac{a^2}{4}},\ q^{(j-\frac{m-n}{2})^2
-bj}=q^{(j-\frac{b}{2}-\frac{m-n}{2})^2-\frac{b(m-n)}{2}-\frac{b^2}{4}},
\end{align}
we have
\begin{align}\label{chi5-start}
    \chi_5(q^a,q^{-a},q^b,q^{-b},q^c;q)&=\frac{(-q^{\frac{1}{2}+c};q)_{\infty}}{(q;q)_{\infty}^2}\sum_{m,n\geq0}\frac{q^{\frac{1}{4}m^2+\frac{1}{4}(m-n)^2+\frac{(b-a)m}{2}+(c-\frac{b}{2})n-\frac{a^2+b^2}{4}}}{(q;q)_m(q;q)_n} \nonumber \\
    &\quad \times \sum_{i\in\mathbb{Z}}q^{(i+\frac{m-a}{2})^2} \sum_{j\in\mathbb{Z}}q^{(j+\frac{n-m-b}{2})^2}.
\end{align}
Let $a,b,c\in \mathbb{Z}$ and
\begin{align}
&  U_{mn}(r_1,r_2):= \frac{q^{m^2+(m-n)^2+r_1m+r_2n}}{(q^4;q^4)_m(q^4;q^4)_n},~~ T_{mn}:=U_{mn}(2b-2a,4c-2b), \\
&I(s,t):=\left\{(m,n)\in \mathbb{Z}_{\geq 0}^2|m\equiv s \!\!\pmod{2}, n\equiv t \!\! \pmod{2}\right\}.
\end{align}
The identity \eqref{chi5-start} with $q$ replaced by $q^4$ can be rewritten as
\begin{align}\label{chi5-abc-H}
& \chi_5(q^{4a},q^{-4a},q^{4b},q^{-4b},q^{4c};q^4)=q^{-a^2-b^2}\frac{(-q^{2+4c};q^4)_{\infty}}{(q^4;q^4)_{\infty}^2} \times \Big(\sum_{(m,n)\in I(a,a+b)} T_{mn}\theta_0^2 \nonumber \\
&\quad +\sum_{(m,n)\in I(a+1,a+b)} T_{mn}\theta_1^2 +\sum_{(m,n) \in I(a,a+b+1)}T_{mn}\theta_0\theta_1 +\sum_{(m,n)\in I(a+1,a+b+1)} T_{mn}\theta_1\theta_0 \Big) \nonumber \\
&=q^{-a^2-b^2}\frac{(-q^{2+4c};q^4)_{\infty}}{(q^4;q^4)_{\infty}^2} \times \Big(\sum_{(m,n)\in I(a,a+b)} T_{mn}\theta_0^2 +\sum_{(m,n)\in I(a+1,a+b)} T_{mn}\theta_1^2 \nonumber \\
&\quad +\sum_{\substack{m,n\geq 0 \\ n\equiv a+b+1 \!\!\! \pmod{2}}}T_{mn}\theta_0\theta_1\Big).
\end{align}
Here $\theta_0$ and $\theta_1$ are defined in \eqref{theta0-defn} and \eqref{theta1-defn}, respectively.

We define
\begin{align}
    H_k(q):=\sum_{m,n\geq0}\frac{q^{m^2+(m-n)^2+(m,n)\gamma_k}}{(q^4;q^4)_m(q^4;q^4)_n}, \quad \widetilde{H}_k(q):=\sum_{m,n\geq0}\frac{(-1)^mq^{m^2+(m-n)^2+(m,n)\gamma_k}}{(q^4;q^4)_m(q^4;q^4)_n}
    \end{align}
where
\begin{align}\label{5-gamma}
\gamma_1=(0,0)^\mathrm{T},\quad \gamma_2=(0,2)^\mathrm{T}, \quad \gamma_3=(-2,2)^\mathrm{T}.
\end{align}
For convenience, we write $\gamma_k=(a_k,b_k)^\mathrm{T}$ ($k=1,2,3$). Summing over $m$ first using \eqref{euler-2}, we deduce that
\begin{align}
    &H_k(q)=\sum_{n\geq 0} \frac{(-q^{2+a_k-2n};q^4)_\infty q^{n^2+b_kn}}{(q^4;q^4)_n}=H_k^{(0)}(q)+H_k^{(1)}(q), \\
    &\widetilde{H}_k(q)=\sum_{n\geq 0} \frac{(q^{2+a_k-2n};q^4)_\infty q^{n^2+b_kn}}{(q^4;q^4)_n}=\widetilde{H}_k^{(0)}(q)+\widetilde{H}_k^{(1)}(q)
\end{align}
where $H_k^{(i)}(q)$ and $\widetilde{H}_k^{(i}(q)$ ($i=0,1$) correspond to the sums over nonnegative integers $n$ satisfying $n\equiv i$ (mod 2), namely,
\begin{align}
H_k^{(i)}(q)&:=\sum_{m,n\geq0}\frac{q^{m^2+(m-2n-i)^2+(m,2n+i)\gamma_k}}{(q^4;q^4)_m(q^4;q^4)_{2n+i}} \nonumber \\
&=\sum_{n\geq 0} \frac{(-q^{2+a_k-2(2n+i)};q^4)_\infty q^{(2n+i)^2+b_k(2n+i)}}{(q^4;q^4)_{2n+i}}, \label{Hki} \\
\widetilde{H}_k^{(i)}(q)&:=\sum_{m,n\geq0}\frac{(-1)^mq^{m^2+(m-2n-i)^2+(m,2n+i)\gamma_k}}{(q^4;q^4)_m(q^4;q^4)_{2n+i}} \nonumber \\
&=\sum_{n\geq 0} \frac{(q^{2+a_k-2(2n+i)};q^4)_\infty q^{(2n+i)^2+b_k(2n+i)}}{(q^4;q^4)_{2n+i}}. \label{wHki}
\end{align}
Clearly for $s,t\in \{0,1\}$ we have
\begin{align}
\sum_{(m,n)\in I(s,t)}  U_{mn}(a_k,b_k)=\frac{1}{2}\Big(H_k^{(t)}(q)+(-1)^s\widetilde{H}_k^{(t)}(q)\Big).\label{rank5-Ust}
\end{align}

\begin{lemma}\label{lem-H}
\label{55}
We have
    \begin{align}
        H_1^{(0)}(q)&=\frac{(-q^2;q^4)_{\infty}^2(q^6,q^{22},q^{28};q^{28})_{\infty}(q^{16},q^{40};q^{56})_{\infty}}{(q^4;q^4)_{\infty}}, \label{H10-product}\\
        H_1^{(1)}(q)&=\frac{2q(-q^4;q^4)_{\infty}^2(q^8,q^{20},q^{28};q^{28})_{\infty}(q^{12},q^{44};q^{56})_{\infty}}{(q^4;q^4)_{\infty}}, \label{H11-product}\\
        H_2^{(0)}(q)&=\frac{(-q^2;q^4)_{\infty}^2(q^2,q^{26},q^{28};q^{28})_{\infty}(q^{24},q^{32};q^{56})_{\infty}}{(q^4;q^4)_{\infty}}, \label{H20-product}\\
        H_2^{(1)}(q)&=\frac{2q^3(-q^4;q^4)_{\infty}^2(q^{12},q^{16},q^{28};q^{28})_{\infty}(q^4,q^{52};q^{56})_{\infty}}{(q^4;q^4)_{\infty}}, \label{H21-product}\\
        H_3^{(0)}(q)&=\frac{2(-q^4;q^4)_{\infty}^2(q^4,q^{24},q^{28};q^{28})_{\infty}(q^{20},q^{36};q^{56})_{\infty}}{(q^4;q^4)_{\infty}}, \label{H30-product}\\
        H_3^{(1)}(q)&=\frac{q(-q^2;q^4)_{\infty}^2(q^{10},q^{18},q^{28};q^{28})_{\infty}(q^8,q^{48};q^{56})_{\infty}}{(q^4;q^4)_{\infty}}. \label{H31-product}
    \end{align}
\end{lemma}
\begin{proof}
From \eqref{Hki} and \eqref{wHki} and using \eqref{aq-finite} we deduce that
\begin{align}
    H_1^{(0)}(q)&=(-q^2;q^4)_{\infty}\sum_{n\geq0}\frac{q^{2n^2}(-q^2;q^4)_{n}}{(q^4;q^4)_{2n}}, \label{H10-mid}\\
    H_1^{(1)}(q)&=(-1;q^4)_{\infty}\sum_{n\geq0}\frac{q^{2n^2+2n+1}(-q^4;q^4)_{n}}{(q^4;q^4)_{2n+1}}, \label{H11-mid}\\
    H_2^{(0)}(q)&=(-q^2;q^4)_{\infty}\sum_{n\geq0}\frac{q^{2n^2+4n}(-q^2;q^4)_{n}}{(q^4;q^4)_{2n}}, \label{H20-mid} \\
    H_2^{(1)}(q)&=(-1;q^4)_{\infty}\sum_{n\geq0}\frac{q^{2n^2+6n+3}(-q^4;q^4)_{n}}{(q^4;q^4)_{2n+1}}, \label{H21-mid} \\
    H_3^{(0)}(q)&=(-1;q^4)_{\infty}\sum_{n\geq0}\frac{q^{2n^2+2n}(-q^4;q^4)_{n}}{(q^4;q^4)_{2n}}, \label{H30-mid} \\
    H_3^{(1)}(q)&=(-q^2;q^4)_{\infty}\sum_{n\geq0}\frac{q^{2n^2+4n+1}(-q^2;q^4)_{n+1}}{(q^4;q^4)_{2n+1}}. \label{H31-mid}
\end{align}
Substituting \eqref{S.80}--\eqref{S.119} into \eqref{H10-mid}--\eqref{H31-mid} we get the desired identities.
\end{proof}
From \eqref{Hki} and \eqref{wHki} it is easy to get the byproducts:
\begin{align}
    \widetilde{H}_1^{(0)}(q)&=H_1^{(0)}(\sqrt{-1}q)=\frac{(q^2;q^4)_{\infty}^2(-q^6,-q^{22},q^{28};q^{28})_{\infty}(q^{16},q^{40};q^{56})_{\infty}}{(q^4;q^4)_{\infty}}, \label{wH10-product}\\
    \widetilde{H}_2^{(0)}(q)&=H_2^{(0)}(\sqrt{-1}q)=\frac{(q^2;q^4)_{\infty}^2(-q^2,-q^{26},q^{28};q^{28})_{\infty}(q^{24},q^{32};q^{56})_{\infty}}{(q^4;q^4)_{\infty}},  \label{wH20-product}\\
    \widetilde{H}_3^{(1)}(q)&=\sqrt{-1}H_3^{(1)}(\sqrt{-1}q)\nonumber \\
    &=-q\frac{(q^2;q^4)_{\infty}^2(-q^{10},-q^{18},q^{28};q^{28})_{\infty}(q^8,q^{48};q^{56})_{\infty}}{(q^4;q^4)_{\infty}}. \label{wH31-product}
\end{align}

Now we are able to state identities for the rank five tadpole Nahm sums. We find nine modular cases.
\begin{theorem}\label{thm-rank5}
Let $S(a,b,c)=\chi_5(q^{4a},q^{-4a},q^{4b},q^{-4b},q^{4c};q^4)$. We have
    \begin{align}
       & S(0,0,0)=4q^2\frac{J_8^6J_{8,28}J_{12,56}}{J_2J_4^5J_8J_{56}}-2q\frac{J_4J_8J_{6,28}J_{16,56}}{J_2^3J_{56}}\nonumber\\
       &\qquad +\frac{1}{2}\frac{J_2^7J_{6,28}J_{16,56}}{J_1^4J_4J_8^3J_{56}}+\frac{1}{2}\frac{J_2^5J_{28}^2J_{12,56}J_{16,56}}{J_4^5J_8J_{56}^2J_{6,28}}, \label{eq-rank5-1} \\
        &S(1,1,1/2)=4\frac{J_8^7J_{8,28}J_{12,56}}{J_4^8J_{56}}-2q^{-1}\frac{J_8^3J_{6,28}J_{16,56}}{J_2^2J_4^2J_{56}}\nonumber\\
        &\qquad +\frac{1}{2}q^{-2}\frac{J_2^8J_{6,28}J_{16,56}}{J_1^4J_4^4J_8J_{56}}-\frac{1}{2}q^{-2}\frac{J_2^6J_8J_{28}^2J_{12,56}J_{16,56}}{J_{4}^8J_{56}^2J_{6,28}}, \label{eq-rank5-2}\\
      &  S(-1,-1,0)=4q^2\frac{J_8^5J_{12,28}J_{4,56}}{J_2J_4^5J_{56}}-2q^{-1}\frac{J_4J_8J_{2,28}J_{24,56}}{J_2^3J_{56}}\nonumber\\
      &\qquad +\frac{1}{2}q^{-2}\frac{J_2^7J_{2,28}J_{24,56}}{J_1^4J_4J_8^3J_{56}}-\frac{1}{2}q^{-2}\frac{J_2^5J_{28}^2J_{4,56}J_{24,56}}{J_4^5J_8J_{56}^2J_{2,28}}, \label{eq-rank5-4}\\
     &  S(0,0,1/2)=4q^4\frac{J_8^7J_{12,28}J_{4,56}}{J_4^8J_{56}}-2q\frac{J_8^3J_{2,28}J_{24,56}}{J_2^2J_4^2J_{56}}\nonumber\\
     &\qquad +\frac{1}{2}\frac{J_2^8J_{2,28}J_{24,56}}{J_1^4J_4^4J_8J_{56}}+\frac{1}{2}\frac{J_2^6J_8J_{28}^2J_{4,56}J_{24,56}}{J_4^8J_{56}^2J_{2,28}},  \label{eq-rank5-5}\\
        &S(0,-1,0)=4\frac{J_8^5J_{4,28}J_{20,56}}{J_2J_4^5J_{56}}-2q\frac{J_4J_8J_{10,28}J_{8,56}}{J_2^3J_{56}}\nonumber\\
        &\qquad +\frac{1}{2}\frac{J_2^7J_{10,28}J_{8,56}}{J_1^4J_4J_8^3J_{56}}-\frac{1}{2}\frac{J_2^5J_{28}^2J_{8,56}J_{20,56}}{J_4^5J_8J_{56}^2J_{10,28}}, \label{eq-rank5-7}\\
       & S(1,0,1/2)=4\frac{J_8^7J_{4,28}J_{20,56}}{J_4^8J_{56}}-2q\frac{J_8^3J_{10,28}J_{8,56}}{J_2^2J_4^2J_{56}}\nonumber\\
       &\qquad +\frac{1}{2}\frac{J_2^8J_{10,28}J_{8,56}}{J_1^4J_4^4J_8J_{56}}+\frac{1}{2}\frac{J_2^6J_{8}J_{28}^2J_{8,56}J_{20,56}}{J_4^8J_{56}^2J_{10,28}},  \label{eq-rank5-8}\\
         & S(-1,-1,-1/2)=2S(1,1,1/2), \label{eq-rank5-3} \\
         & S(-2,-2,-1/2)=2q^{-8}S(0,0,1/2), \label{eq-rank5-6} \\
         &S(-1,-2,-1/2)=2q^{-4}S(1,0,1/2). \label{eq-rank5-9}
    \end{align}
\end{theorem}
\begin{proof}
Let $(a,b,c)=(0,0,0)$ and $(1,1,1/2)$ in \eqref{chi5-abc-H}. Note that for these choices we have $(2b-2a,4c-2b)^\mathrm{T}=\gamma_1$ (see \eqref{5-gamma}). Using \eqref{rank5-Ust} we have
\begin{align}
   & S(0,0,0)=\frac{(-q^2;q^4)_{\infty}}{(q^4;q^4)_{\infty}^2}\Big(\frac{1}{2}\big(H_1^{(0)}(q)+\widetilde{H}_1^{(0)}(q) \big)\theta_0^2 \nonumber \\
   &\quad +\frac{1}{2}\big(H_1^{(0)}(q)-\widetilde{H}_1^{(0)}(q) \big)\theta_1^2+H_1^{(1)}(q)\theta_0\theta_1\Big),  \label{proof-5-1} \\
    & S(1,1,1/2)=q^{-2}\frac{(-q^4;q^4)_{\infty}}{(q^4;q^4)_{\infty}^2}\Big(\frac{1}{2}\big(H_1^{(0)}(q)-\widetilde{H}_1^{(0)}(q)\big)\theta_0^2 \nonumber \\
   &\qquad +\frac{1}{2}\big(H_1^{(0)}(q)+\widetilde{H}_1^{(0)}(q) \big)\theta_1^2+H_1^{(1)}(q)\theta_0\theta_1\Big). \label{proof-5-2}
\end{align}

Similarly, setting $(a,b,c)$ as $(-1,-1,0)$ and $(0,0,1/2)$  which satisfy $(2b-2a,4c-2b)^\mathrm{T}=\gamma_2$ in \eqref{chi5-abc-H}, we deduce that
\begin{align}
  & S(-1,-1,0)=q^{-2}\frac{(-q^{2};q^4)_{\infty}}{(q^4;q^4)_{\infty}^2}\Big(\frac{1}{2}\big(H_2^{(0)}(q)-\widetilde{H}_2^{(0)}(q)\big)\theta_0^2 \nonumber \\
    & \quad +\frac{1}{2}\big(H_2^0(q)+\widetilde{H}_2^{(0)}(q)\big)\theta_1^2+H_2^{(1)}(q)\theta_0\theta_1\Big), \label{proof-5-4}\\
  &  S(0,0,1/2)=\frac{(-q^4;q^4)_{\infty}}{(q^4;q^4)_{\infty}^2}\Big(\frac{1}{2}\big(H_2^{(0)}(q)+\widetilde{H}_2^{(0)}(q)\big)\theta_0^2 \nonumber\\
    & \quad +\frac{1}{2}\big(H_2^{(0)}(q)-\widetilde{H}_2^{(0)}(q)\big)\theta_1^2 +H_2^{(1)}(q)\theta_0\theta_1\Big). \label{proof-5-5}
\end{align}

Finally, setting $(a,b,c)$ as $(0,-1,0)$ and $(1,0,1/2)$ which satisfy $(2b-2a,4c-2b)^\mathrm{T}=\gamma_3$ in \eqref{chi5-abc-H}, we deduce that
\begin{align}
 &   S(0,-1,0)=q^{-1}\frac{(-q^{2};q^4)_{\infty}}{(q^4;q^4)_{\infty}^2}\Big(\frac{1}{2}\big(H_3^{(1)}(q)+\widetilde{H}_3^{(1)}(q) \big)\theta_0^2 \nonumber\\
    & \quad +\frac{1}{2}\big(H_3^{(1)}(q)-\widetilde{H}_3^{(1)}(q) \big)\theta_1^2+H_3^{(0)}(q)\theta_0\theta_1\Big), \label{proof-5-7}\\
  &  S(1,0,1/2)=q^{-1}\frac{(-q^4;q^4)_{\infty}}{(q^4;q^4)_{\infty}^2}\Big(\frac{1}{2}\big( H_3^{(1)}(q)-\widetilde{H}_3^{(1)}(q)\big)\theta_0^2 \nonumber\\
    &\quad +\frac{1}{2}\big(H_3^{(1)}(q)+\widetilde{H}_3^{(1)}(q) \big)\theta_1^2+H_3^{(0)}(q)\theta_0\theta_1\Big). \label{proof-5-8}
\end{align}
Substituting \eqref{H10-product}--\eqref{H31-product} and \eqref{wH10-product}--\eqref{wH31-product} into \eqref{proof-5-1}--\eqref{proof-5-8}, and using the method in \cite{Frye-Garvan} to verify theta function identities, we prove \eqref{eq-rank5-1}--\eqref{eq-rank5-8}.

If the  vectors $(a_i,b_i,c_i)$ ($i=1,2$) satisfy that $a_i,b_i\in \mathbb{Z}$, $c_1-c_2\in \mathbb{Z}$ and
\begin{align}
   a_1-a_2=b_1-b_2=2(c_1-c_2),
\end{align}
then from \eqref{chi5-abc-H} we see that $S(a_1,b_1,c_1)$ and $S(a_2,b_2,c_2)$ only differ by the leading factors consist of powers of $q$ and some infinite products, namely,
\begin{align}\label{S-abc-relation}
    S(a_1,b_1,c_1)=q^{a_2^2+b_2^2-a_1^2-b_1^2}\frac{(-q^{2+4c_1};q^4)_\infty}{(-q^{2+4c_2};q^4)_\infty}S(a_2,b_2,c_2).
\end{align}
This proves \eqref{eq-rank5-3}--\eqref{eq-rank5-9}.
\end{proof}

Based on the identities in Theorem \ref{thm-rank5}, it is easy to check that $q^CS(a,b,c)$ is modular for the choices of $(a,b,c,C)$ given in Table \ref{tab-rank5}.
\begin{table}[H]
\renewcommand{\arraystretch}{1.5}
    \centering
    \begin{tabular}{c|ccccc}
    \hline
        $(a,b,c)$ & $(0,0,0)$  & $(1,1,1/2)$ &
        $(-1,-1,0)$ & $(0,0,1/2)$ & $(0,-1,0)$  \\
        \hline
        $C$ & $-55/84$  & $67/42$  & $137/84$  & $-5/42$ & $65/84$ \\
        \hline
        $(a,b,c)$  & $(1,0,1/2)$ & $(-1,-1,-2)$ & $(-2,-2,-1/2)$  & $(-1,-2,-1/2)$  & \\
        \hline
        $C$ & $43/42$ & $67/42$  & $331/42$ & $211/42$  & \\
        \hline
    \end{tabular}
    \caption{Values of $(a,b,c,C)$ such that $q^CS(a,b,c)$ is modular}
    \label{tab-rank5}
    \end{table}

\section{Concluding remarks}\label{sec-remark}
It is our hope that the rank reduction formulas \eqref{eq-reduce-even} and \eqref{eq-reduce-odd} may shed some light to solve higher rank cases. The computations get much more complicated as the rank increases. In fact,  when dealing with the rank six and seven cases, we already face difficulties in evaluating certain rank four and rank three Nahm-type sums. We will not discuss them here. Instead, we conclude this paper by presenting briefly a new proof to the rank three case which was first proved by Milas and Wang \cite{MW24}.

Setting $r=1$ in \eqref{eq-reduce-odd} we have
\begin{align}
    \chi_3(x_1,1/x_1,x_3;q)=\frac{(-q^{\frac{1}{2}x_3};q)_\infty}{(q;q)_\infty} \sum_{n=0}^\infty \frac{q^{\frac{1}{4}n^2}}{(q;q)_n}\big(\frac{x_3}{x_1}\big)^n \Big(\sum_{\ell \in \mathbb{Z}} q^{(\ell+\frac{n}{2})^2}x_1^{-\ell}\Big). \label{eq-rank3-start}
\end{align}
In particular, when $x_1=q^a$ ($a\in \mathbb{Z}$) and $x_3=q^c$ we have
\begin{align}
   & \chi_3(q^a,q^{-a},q^c;q)=q^{-\frac{1}{4}a^2}\frac{(-q^{\frac{1}{2}+c};q)_\infty}{(q;q)_\infty} \sum_{n=0}^\infty \frac{q^{\frac{1}{4}n^2+(c-\frac{1}{2}a)n}}{(q;q)_n} \sum_{\ell \in \mathbb{Z}}q^{(\ell+\frac{1}{2}n-\frac{1}{2}a)^2} \nonumber \\
    &=q^{-\frac{1}{4}a^2}\frac{(-q^{\frac{1}{2}+c};q)_\infty}{(q;q)_\infty}  \Big(\theta_0(q^{\frac{1}{4}})\sum_{\substack{n\in \mathbb{Z}_{\geq 0} \\ n\equiv a \!\!\! \pmod{2}}} \frac{q^{\frac{1}{4}n^2+(c-\frac{1}{2}a)n}}{(q;q)_n} \nonumber \\
    &\qquad \qquad +\theta_1(q^{\frac{1}{4}})\sum_{\substack{n\in \mathbb{Z}_{\geq 0} \\ n\equiv a+1 \!\!\! \pmod{2}}} \frac{q^{\frac{1}{4}n^2+(c-\frac{1}{2}a)n}}{(q;q)_n} \Big). \label{eq-rank3-mid}
\end{align}
Recall the following identities of Rogers (see \cite[pp.\ 330--332]{Rogers1894} and \cite[p.\ 330 (3), 2nd Eq.]{Rogers1917}):
\begin{align}
&\sum_{\substack{n\in \mathbb{Z}_{\geq 0} \\ n\equiv 0 \!\!\! \pmod{2}}}\frac{q^{\frac{1}{4}n^2}}{(q;q)_n}=\sum_{n=0}^\infty \frac{q^{n^2}}{(q;q)_{2n}}=\frac{(q^2,q^8,q^{10};q^{10})_\infty (q^6,q^{14};q^{20})_\infty}{(q;q)_\infty}, \quad  \label{S79}\\
&\sum_{\substack{n\in \mathbb{Z}_{\geq 0} \\ n\equiv 1 \!\!\! \pmod{2}}}\frac{q^{\frac{1}{4}n^2}}{(q;q)_n}=q^{\frac{1}{4}}\sum_{n=0}^\infty \frac{q^{n^2+n}}{(q;q)_{2n+1}}=q^{\frac{1}{4}}\frac{(q^3,q^7,q^{10};q^{10})_\infty (q^4,q^{16};q^{20})_\infty}{(q;q)_\infty},  \quad  \label{S94} \\
&\sum_{\substack{n\in \mathbb{Z}_{\geq 0} \\ n\equiv 0 \!\!\! \pmod{2}}}\frac{q^{\frac{1}{4}n^2+\frac{1}{2}n}}{(q;q)_n}=\sum_{n=0}^\infty \frac{q^{n^2+n}}{(q;q)_{2n}}=\frac{(q,q^9,q^{10};q^{10})_\infty (q^8,q^{12};q^{20})_\infty}{(q;q)_\infty}, \quad  \label{S99} \\
&\sum_{\substack{n\in \mathbb{Z}_{\geq 0} \\ n\equiv 1 \!\!\! \pmod{2}}}\frac{q^{\frac{1}{4}n^2+\frac{1}{2}n}}{(q;q)_n}=q^{\frac{3}{4}}\sum_{n=0}^\infty \frac{q^{n^2+2n}}{(q;q)_{2n+1}}=q^{\frac{3}{4}}\frac{(q^4,q^6,q^{10};q^{10})_\infty (q^2,q^{18};q^{20})_\infty}{(q;q)_\infty}. \label{Rogers-1}
\end{align}

Note that $(-q^{\frac{1}{2}+c};q)_\infty$ is modular after multiplying some powers of $q$ when $c\in \{-\frac{1}{2},0,\frac{1}{2}\}$. The identities \eqref{S79}--\eqref{Rogers-1} suggest that we may choose $c-\frac{1}{2}a\in \{0,\frac{1}{2}\}$ in order to make the right side of \eqref{eq-rank3-mid} being modular up to some power of $q$. Combining these two restrictions together we get six possible modular Nahm sums $\chi_3(q^a,q^{-a},q^c;q)$ with $(a,c)$ being
\begin{align}
    (-1,-1/2),~(-2,-1/2),~(0,0),~(-1,0),~(1,1/2),~(0,1/2).
\end{align}
For each of these cases, substituting \eqref{theta0-defn},  \eqref{theta1-defn} and \eqref{S79}--\eqref{Rogers-1} into \eqref{eq-rank3-mid}, we reprove the identities of Milas and Wang \cite[Theorem 1.2]{MW24} and confirm Conjecture \ref{conj-CMP} in the rank three case. The common feature of the proof in \cite{MW24} and the proof above is that they all use \eqref{S79}--\eqref{Rogers-1}. Note that in \cite{MW24} these six modular cases were discovered by analyzing dual triples of the sixth rank three example in Zagier's list \cite[Table 3]{Zagier}. Our approach is more straightforward in finding these modular cases and the treatment is unified.


\subsection*{Acknowledgements}
This work was supported by the National Key R\&D Program of China (Grant No.\ 2024YFA1014500) and the National Natural Science Foundation of China (Grant No.\ 123B1016).

\end{document}